\documentclass[11pt, reqno]{amsart}
\usepackage[utf8]{inputenc}
\usepackage{amsmath}
\usepackage{amsfonts}
\usepackage{amssymb}
\usepackage{amsthm}
\usepackage{dsfont}
\usepackage{hyperref}
\usepackage{stmaryrd}
\usepackage{enumitem}
\usepackage[linesnumbered,lined,boxed,commentsnumbered]{algorithm2e}
\usepackage{tikz-cd}

\DeclareMathOperator{\Hom}{Hom}
\DeclareMathOperator{\cha}{char}

\DeclareMathOperator{\End}{End}
\DeclareMathOperator{\Spec}{Spec}
\DeclareMathOperator{\Mat}{Mat}

\DeclareMathOperator{\HH}{H}
\DeclareMathOperator{\Gl}{Gl}
\DeclareMathOperator{\ev}{ev}
\DeclareMathOperator{\id}{id}

\DeclareMathOperator{\pr}{pr}

\DeclareMathOperator{\sh}{sh}
\DeclareMathOperator{\CC}{\mathbb{C}}
\DeclareMathOperator{\QQ}{\mathbb{Q}}

\DeclareMathOperator{\ZZ}{\mathbb{Z}}
\newcommand{\tG}[1]{\mathcal{G}(\mathcal{L}^{#1})}
\newcommand{\thetacha}[2]{\vartheta \hspace{-0.5mm}\big[\hspace{-0.5mm}\begin{smallmatrix}
#1\\#2
\end{smallmatrix}\hspace{-0.5mm} \big]}
\theoremstyle{definition}
\newtheorem{definition}{Definition}[section]

\newtheorem{exmpl}[definition]{Example}
\newtheorem{rem}[definition]{Remark}
\theoremstyle{theorem}
\newtheorem{thm}[definition]{Theorem}
\newtheorem{cor}[definition]{Corollary}
\newtheorem{lem}[definition]{Lemma}
\newtheorem{prop}[definition]{Proposition}

\begin{document}
\title{Theta Nullvalues of Supersingular Abelian varieties}
\begin{abstract}
Let $\eta$ be a polarization with connected kernel on a superspecial abelian variety $E^g$. We give a sufficient criterion which allows the computation of the theta nullvalues of any quotient of $E^g$ by a maximal isotropic subgroup scheme of $\ker(\eta)$ effectively.
\par This criterion is satisfied in many situations studied by Li and Oort \cite{Li-Oort}. We used our method to implement an algorithm that computes supersingular curves of genus 3.
\end{abstract}
\author[Pieper]{Andreas Pieper}
\address{
  Andreas Pieper,
  Institut für Algebra und Zahlentheorie,
  Universität Ulm,
  Helmholtzstrasse 18,
  89081 Ulm,
  Germany
}
\email{andreas.pieper@uni-ulm.de}

\maketitle
\tableofcontents
\section{Introduction}
Recall that an Abelian variety of dimension $g$ is called \emph{supersingular} if it becomes isogenous to a product of supersingular elliptic curves after an extension of the ground field. For $g=1$ there are only finitely many isomorphism classes of supersingular Abelian varieties over a given algebraically closed field. Then for $g>1$ a new phenomenon arises: There are non-constant families of supersingular Abelian varieties over a positive dimensional base.
\par The first investigation of this phenomenon was \cite{Moret-Bailly}. Moret-Bailly constructed families $\mathcal{A}\rightarrow \mathbb{P}^1$ of principally polarized supersingular Abelian surfaces. Later Oort proved that every principally polarized supersingular Abelian surface is the fiber of one such family. This proves that the locus of supersingular points in the Siegel moduli space is $1$-dimensional.
\par Together with Li, Oort \cite{Li-Oort} generalized these results to arbitrary $g$: They constructed families $\mathcal{A}\rightarrow S$ of $g$-dimensional supersingular p.p.a.v.s over a $\lfloor \frac{g^2}{4} \rfloor$-dimensional base $S$. They proved that every supersingular p.p.a.v. is the fiber of such a family. Furthermore, the induced map from $S$ to the Siegel moduli space is generically finite. Therefore the locus of supersingular points in the Siegel moduli space is $\lfloor \frac{g^2}{4} \rfloor$-dimensional.
\par The goal of this paper is to make these families explicit. Let us fix once and for all a supersingular elliptic curve $E$ defined over $\mathbb{F}_p$ with $p>2$ such that its Frobenius satisfies $F^2+p=0$. The construction of Li and Oort starts with a polarization $\eta$ on $E^g$ such that $\ker(\eta)=E^g[F^{g-1}]$. The base scheme $S$ is a parameter space for certain maximal isotropic subgroup schemes $\mathcal{H}\subseteq \ker(\eta)$ described via Dieudonn\'e theory. The fiber above a point of $S$ will be $E^g/\mathcal{H}$ equipped with the polarization descended from $\eta$.
\par This motivates the following setup: We can allow ourselves to be slightly more general and will assume that $\eta$ is a polarization on $E^g$ such that $\ker(\eta)$ is any connected group scheme. Furthermore $\mathcal{H}\subseteq \ker(\eta)$ shall be a maximal isotropic subgroup scheme. Denote by $A$ the quotient $E^g/\mathcal{H}$ and by $\mathcal{L}$ a symmetric line bundle on $A$ that induces the polarization descended from $\eta$. In this article we will give a method that calculates the algebraic theta nullvalues of $(A, \mathcal{L}^4)$. The method generalizes to $(A,\mathcal{L}^N)$ for $N$ arbitrary but we will confine ourselves to $N=4$ to simplify the exposition. The restriction to $N=4$ is interesting enough because the theta nullvalues of $(A, \mathcal{L}^4)$ uniquely determine a principally polarized abelian variety (a consequence of \cite[\S 4, Theorem 2]{EqDefAV}).
\par In characteristic $0$, theta nullvalues are computed analytically by evaluating Riemann theta functions. But, this is not available for us because we are in characteristic $p>0$ and our setup will lift very rarely. Instead, we devise a purely algebraic machinery to calculate theta nullvalues.
\par This is based on Mumford's theory of theta groups. To the Abelian variety $A$ and the line bundle $\mathcal{L}^4$ one can associate a group $\tG{4}$ that naturally acts on $\HH^0(A, \mathcal{L}^4)$. The group action gives a natural basis of $\HH^0(A, \mathcal{L}^4)$ after choosing a suitable level structure on $A$, called $\Gamma(4,8)$-level structure. The algebraic theta nullvalues are defined to be the evaluation at $0$ of this natural basis. Via the  bijection between line bundles and divisors modulo rational equivalence the action of the theta group has an explicit description in terms of rational functions on $A$. However, we do not know $A$ yet. Therefore we will instead work with the pullback of these rational functions along $E^g\rightarrow E^g/\mathcal{H}$.
\par Indeed, let $\Theta \subset A$ be the divisor of vanishing of the unique (up to a scalar) non-zero section of $\mathcal{L}$. We will chose a fixed divisor $D\subset E^g$ (see below for the precise definition) in the rational equivalence class of $\pi^{-1}(\Theta)$. The first rational function on $E^g$ that we need to figure out is $\rho_{\mathcal{H}}$ such that $(\rho_{\mathcal{H}})=\pi^{-1}(\Theta)-D$. In the case of Moret-Bailly families this kind of problem was solved in \cite[Theorem 4.18]{AP}. The main idea of the quoted paper was to exploit the fibers of the family that are as simple as possible. These are a product of elliptic curves, equipped with the product polarization.
\par This idea will also play a prominent role in the present paper. We define a \emph{completely decomposed $\Theta_{\eta}$-divisor} to be a divisor inducing $\eta$ that is the preimage of the natural theta divisor along an isogeny $E^g\rightarrow E_1 \times \ldots \times E_g$. The divisor $D$ fixed above will be a suitable translate of one completely decomposed $\Theta_{\eta}$-divisor.
\par In \cite{AP} the author uses a second completely decomposed $\Theta_{\eta}$-divisor to construct an isomorphism of $\mathcal{G}(\mathcal{O}(D))$ with a Heisenberg group. In our general setting $\mathcal{G}(\mathcal{O}(D))$ will usually not be a Heisenberg group. Nevertheless, it will turn out to be enough for the present method to have level groups in $\mathcal{G}(\mathcal{O}(D))$ that generate $\ker(\eta)$ in a certain sense made precise in the main body of the article. This leads to a technical restriction on $\eta$ as we will assume that there are sufficiently many completely decomposed $\Theta_{\eta}$-divisors. Under this assumption we can give a method for calculating $\rho_{\mathcal{H}}$ which is explained in Section \ref{SecInvCon}.
\par There is a second kind of rational functions that we need for computing the theta nullvalues: To make sense of theta nullvalues one needs a theta structure on $\tG{4}$ and thus two maximal level groups in $\tG{4}$. To construct these we will define two level groups in $\mathcal{G}(\mathcal{O}(4D))$ that will descend to the desired maximal level groups in $\tG{4}$. This will be explained in section \ref{SecEtaleLevGr}.
\par The paper is organized as follows: Chapter \ref{ChapThetaGr} recalls Mumford's theory of theta groups. Contrary to \cite{EqDefAV} we do not assume that the ground field $k$ is algebraically closed. Instead, we will always indicate the condition on $k$ such that everything works without a field extension. This convention is used throughout the main body of the paper. The reader may however safely assume that $k$ is algebraically closed if he or she prefers that. For sake of simplicity we shall assume that $k$ is algebraically closed for the rest of the introduction.
\par The most important definition in Chapter \ref{ChapThetaGr} is a so-called $\Gamma(4,8)$-level structure on a principally polarized abelian variety. The terminology stems from the fact that a $\Gamma(4,8)$-level induces a full level $4$ structure and is induced by a full level $8$ structure. Next we will explain the compatibility with the complex analytic theory of Riemann theta functions. Notice that over $\mathbb{C}$ the moduli space of p.p.a.v. with a $\Gamma(4,8)$-level is isomorphic to $\mathfrak{H}_g/\Gamma(4,8)$ where $\mathfrak{H}_g$ denotes the Siegel upper half space and
$$\Gamma(4,8)=\left\lbrace X=\begin{pmatrix}
\alpha & \beta\\
\gamma & \delta
\end{pmatrix}\in \text{Sp}_{2g}(\mathbb{Z}) | X\equiv I_{2g} \mod 4,\, \text{diag}(\alpha^t \beta)\equiv\text{diag}(\gamma^t \delta)\equiv 0 \mod 8  \right\rbrace$$
It is a fact that Riemann theta nullvalues with half-integer characteristics define Siegel modular forms of weight $1/2$ and level $\Gamma(4,8)$. The theta nullvalues of $(A, \mathcal{L}^4)$ are equal to these analytic Riemann theta nullvalues up to a linear transformation.
\par In chapter \ref{Chapfcgrschem} we discuss finite commutative group schemes, their representation theory, and Dieudonn\'e theory. We use Dieudonn\'e theory to reduce all calculations with group schemes to truncated Witt group schemes. Their explicit Hopf algebra structure and the explicit description of endomorphisms will be a useful tool.
\par In chapter \ref{ChapcdTheta} we introduce completely decomposed $\Theta_\eta$-divisors and their basic properties. Then we are ready to present the main method in chapter \ref{ChapMain}. In the last chapter we give some examples: We implemented an algorithm that computes the generic fiber of Moret-Bailly families. The speed of the computation was satisfying but can be shown to be not optimal for this particular problem.
\par On the other hand we implemented the method for $g=3$. We used the algorithm to compute equations for some genus $3$ curves with supersingular Jacobian. All the source code related to this paper is available at \url{https://github.com/Andreas-Pieper/SuperTheta.git}\\
\textbf{Acknowledgements:} This article is part of the author's PhD thesis under the supervision of Prof. Wewers. I thank him for his patience and his support. Furthermore, I thank J. Sijsling for providing the source code \cite{PlaneCM} for our implementation. The author is indebted to the anonymous referee for numerous suggestions and corrections.
\section{Theta Groups}\label{ChapThetaGr}
In this chapter we recall Mumford's theory of theta groups. The theory is applied twice in this paper: First we use it to study the isogeny $E^g\rightarrow E^g/\mathcal{H}$. Secondly it is needed for the definition of algebraic theta nullvalues and in particular also for their computation. The fact that $E^g\rightarrow E^g/\mathcal{H}$ is purely inseparable forces us to use the general (non)-separable theory and thus we need to define the theta group as a group scheme. Everything in this chapter is well-known and mainly due to Mumford.
\begin{definition}
Let $A/k$ be an abelian variety and $\mathcal{L}$ a line bundle. Then we define the \emph{theta group of $\mathcal{L}$} via the functor of points
\begin{equation*}
\begin{aligned}
\mathcal{G}(\mathcal{L}): \{\text{Sch} /k\}\longrightarrow \{\text{Groups}\}\\
T\mapsto \Big\lbrace(x,\phi)| x\in A(T)\textrm{ and } & \phi: \mathcal{L}_T \stackrel{\sim}{\longrightarrow} t_x^* \mathcal{L}_T\\
& \text{ an isomorphism of line bundles on } A\times T \Big\rbrace
\end{aligned}
\end{equation*}
The group structure is defined by $(x_1,\phi_1)\cdot (x_2,\phi_2)=(x_1+x_2,t_{x_2}^*\phi_1\circ \phi_2)$.
\end{definition}
The functor $\mathcal{G}(\mathcal{L})$ is represented by a group scheme over $k$ \cite[Lemma 8.2]{vdGMooAV}.
\begin{definition} Let $A$ be an abelian variety and $\mathcal{L}$ an ample line bundle. For an integer $n\in \mathbb{N}$ with $\cha(k) \nmid n$ we define the homomorphism
$$\varepsilon_n: \mathcal{G}(\mathcal{L})\rightarrow \mathcal{G}(\mathcal{L}^n)$$
as follows: $\varepsilon_n(x, \phi)=(x, \phi^n)$ where $\phi^n$ denotes the composition
$$\mathcal{L}^n\stackrel{\phi^{\otimes n}}{\longrightarrow} (t_x^*\mathcal{L})^{\otimes n}\stackrel{\sim}{\rightarrow} t_x^* \mathcal{L}^{n}$$
\end{definition}
We denote by $K(\mathcal{L})$ the image of $\mathcal{G}(\mathcal{L})\rightarrow A,\, (x, \phi)\mapsto x$. The relationship between $K(\mathcal{L})$ and $K(\mathcal{L}^n)$ is as follows:
\begin{lem}\label{KMultn}
$K(\mathcal{L}^n)=[n]^{-1}(K(\mathcal{L}))$
\end{lem}
\begin{proof}
See \cite[p. 310]{EqDefAV}.
\end{proof}
\subsection{Symmetric line bundles}
Recall that a line bundle $\mathcal{L}$ on an abelian variety is called symmetric if $[-1]^*\mathcal{L}\cong\mathcal{L}$. A normalized isomorphism of $\mathcal{L}$ and $[-1]^*\mathcal{L}$ is an isomorphism $\phi:\mathcal{L} \stackrel{\sim}{\rightarrow} [-1]^* \mathcal{L}$ such that $\phi_{|0}: \mathcal{L}_{|0}\rightarrow [-1]^* \mathcal{L}_{|0}=\mathcal{L}_{|0} $ is the identity. There is a unique normalized isomorphism of $\mathcal{L}$ and $[-1]^*\mathcal{L}$.
\par A symmetric line bundle defines a quadratic function on the 2-torsion as follows:
\begin{definition} Let $A$ be an abelian variety and $\mathcal{L}$  a symmetric line bundle. Denote by $\phi: \mathcal{L}\stackrel{\sim}{\rightarrow} [-1]^* \mathcal{L}$ the normalized isomorphism. Define
$$e_{*}^\mathcal{L}: A[2](\overline{k})\rightarrow \overline{k}^\times$$
as follows: For a scheme valued point $x\in A[2](\overline{k})$ we define $e_{*}^\mathcal{L}(x)=\alpha$ where $\alpha\in \overline{k}^\times$ is the scalar such that
$$\phi_{|x}:\mathcal{L}_{|x}\rightarrow ([-1]^*\mathcal{L})_{|x}=\mathcal{L}_{|-x}=\mathcal{L}_{|x}$$
is multiplication by $\alpha$.
\end{definition}
This function $e_*$ as the following properties:
\begin{prop}\label{QFProp} \begin{itemize}
\item[i)] $e_{*}^\mathcal{L}$ has values in $\pm 1$
\item[ii)] $e_{*}^\mathcal{L}$ is a quadratic function with associated bilinear form $e^{\mathcal{L}^2}$, i.e. for all $x,y\in A[2](\overline{k})$ one has
$$e_*^\mathcal{L}(x+y)=e_*^\mathcal{L}(x)\cdot e_*^\mathcal{L}(y)\cdot e^{\mathcal{L}^2}(x,y)$$
\item[iii)] $e_{*}^\mathcal{L}$ is functorial: If $\pi:A\rightarrow B$ is a homomorphism and $\mathcal{L}$ is a symmetric line bundle on $B$ then for all $x\in A[2](\overline{k})$:
$$e_*^{\pi^*\mathcal{L}}(x)=e_*^\mathcal{L}(\pi(x))$$
\item[iv)] If $D$ is a symmetric divisor with $\mathcal{L}\cong \mathcal{O}(D)$. Then for all $x\in A[2](\overline{k})$ one has
$$e_*^\mathcal{L}(x)=(-1)^{m(x)-m(0)}$$
where for $y\in A$ we denote by $m(y)$ the multiplicity $D$ at $y$, i.e. $m(y)=0$ if $y\notin D$ otherwise $m(y)$ is the multiplicity of the local ring $\mathcal{O}_{D,y}$.
\end{itemize}
\end{prop}
\begin{proof}
See \cite[p. 304, First Properties i), iii)]{EqDefAV} for i) and iii). Assertion ii) is \cite[p. 314, Corollary 1]{EqDefAV}. For iv) we refer to \cite[p. 307, Proposition 2]{EqDefAV}.
\end{proof}
For a symmetric line bundle $\mathcal{L}$ Mumford defines two homomorphism $\delta_{-1}: \mathcal{G}(\mathcal{L})\rightarrow \mathcal{G}(\mathcal{L})$ and $\eta_n: \mathcal{G}(\mathcal{L}^n)\rightarrow \mathcal{G}(\mathcal{L})$ for any $n\in \mathbb{N}$.
\begin{definition}
Let $\psi: \mathcal{L}\stackrel{\sim}{\rightarrow} [-1]^*\mathcal{L}$ be an isomorphism. We define the map
$$\delta_{-1}: \tG{}\rightarrow \tG{}$$
$$\delta_{-1}(x, \phi)=(-x, (t_{-x}^* \psi) \circ ([-1]^* \phi) \circ \psi) $$
Where the latter is the following composition:
$$\begin{tikzcd}
\mathcal{L} \arrow{r}{\psi}& \left[-1\right]^* \mathcal{L} \arrow{r}{\left[-1\right]^* \phi}  & \left[-1\right]^* t_x^* \mathcal{L} \arrow[equal]{r}&  t_{-x}^* \left[-1\right]^*\mathcal{L} & t_{-x}^* \mathcal{L} \arrow{l}{t_{-x}^* \psi}
\end{tikzcd}$$
\end{definition}
The definition of $\eta_ n$ is as follows:
\begin{definition}\label{Defeta} Since $\mathcal{L}$ is symmetric, there is an isomorphism $\psi: \mathcal{L}^{n^2}\stackrel{\sim}{\rightarrow  }[n]^* \mathcal{L}$. 
We define
$$\eta_n: \mathcal{G}(\mathcal{L}^n) \rightarrow \mathcal{G}(\mathcal{L}),\, \eta_n(x, \phi)=(n x, \rho)$$

where $\rho$ is the unique \footnote{Uniqueness is clear and existence follows from Lemma \ref{KMultn}} isomorphism $\rho: \mathcal{L}\stackrel{\sim}{\rightarrow} t_{nx} \mathcal{L}$ such that the diagram
$$ \begin{tikzcd}
\mathcal{L}^{n^2} \arrow["\phi^n",r]\arrow["\psi",dd]& t_x^*\mathcal{L}^{n^2}\arrow["t_x^* \psi", d]\\
 & t_x^*(\left[n\right]^* \mathcal{L}) \arrow[equal]{d}\\
\left[n\right]^*\mathcal{L} \arrow{r}[swap]{\left[n \right]^*\rho}& \left[ n \right]^* t_{nx}^* \mathcal{L}
\end{tikzcd}$$
commutes.
\end{definition}
It is clear that the definition of $\eta_n$ is independent of the choice of $\psi$.
\par The following formula of Mumford shows that the map $\eta_n$ captures the quadratic form $e_*$:
\begin{prop}\label{etae*} Let $z\in \mathcal{G}(\mathcal{L}^2)(\overline{k})$ be an element of order $2$ mapping to $x\in K(\mathcal{L}^2)(\overline{k})$. Then $\eta_2(z)$ is in $\overline{k}^\times$ and
$$\eta_2(z)=e^\mathcal{L}_*(x)$$
\end{prop}
\begin{proof}
See \cite[\S 2, Proposition 6]{EqDefAV}.
\end{proof}
\subsection{Theta structures}
Let $A$ be an abelian variety and $\mathcal{L}$ an ample line bundle with $\cha(k) \nmid \deg(\mathcal{L})$. Although it would not be necessary for this section we will stick with the convention that we view the theta group as a group scheme. Otherwise we would have to hop back and forth between statements written in terms of group schemes and its $\overline{k}$-valued points respectively.
\begin{definition}
For a sequence of positive integers $\delta=(d_1,\ldots, d_g)$ we denote $Z_\delta=\bigoplus_{i=1}^g \mathbb{Z}/d_i \mathbb{Z}$.
\par Now we define a group scheme $\mathcal{G}(\delta)$ as follows: As a scheme
$$\mathcal{G}(\delta)=\mathbb{G}_m\times \underline{ Z_\delta} \times \underline{Z_\delta}^D$$
Where $\underline{(\cdot)}$ denotes the constant group scheme. and $(\cdot)^D$ denotes the Cartier dual.
\par The group structure is defined on $T$-valued points by
$$(s, x, \chi)\cdot(s', x', \chi')=(s s' \chi'(x), x+x', \chi+\chi')$$
where we use the identification $\underline{Z_\delta}^D(T)=\Hom_T(\underline{Z_\delta}_T, \mathbb{G}_{m,T} )$.
\end{definition}
\begin{definition} 
 A \emph{theta structure} is an isomorphism $\mathcal{G}(\mathcal{L}) \cong \mathcal{G}(\delta)$ restricting to the identity on $\mathbb{G}_m$.
\end{definition}
\begin{definition}
We define $[2]: Z_\delta\rightarrow Z_{2\delta},\, x \mapsto 2x$ and we denote by $r: Z_{2\delta} \rightarrow Z_{\delta},\, x\mapsto x$ the natural quotient map. We denote by $[2]^D: \underline{Z_{2\delta}}^D\rightarrow \underline{Z_\delta}^D$ and $r^D: \underline{Z_{2\delta}}^D\rightarrow \underline{Z_\delta}^D$ the induced dual maps.
\par We define $E_2:\mathcal{G}(\delta)\rightarrow \mathcal{G}(2 \delta)$ to be the map
$$E_2((s, x , \chi))=(s^2, [2]x, r^D(\chi))$$
\par We define $D_{-1}: \mathcal{G}(\delta)\rightarrow \mathcal{G}(\delta)$ to be the map
$$D_{-1}((s, x , \chi))=(s, -x, -\chi)$$
\par We define $H_2:\mathcal{G}(2\delta)\rightarrow \mathcal{G}(\delta)$ to be the map
$$H_2((s, x , \chi))=(s^2, r(x), [2]^D(\chi))$$
\end{definition}
\begin{definition} Assume in addition that $\mathcal{L}$ is totally symmetric (see \cite[Definition on p. 305]{EqDefAV}). A $\vartheta$-structure
$$f: \mathcal{G}(\mathcal{L})\stackrel{\sim}{\rightarrow} \mathcal{G}(\delta)$$
is called \emph{symmetric} if $f\circ \delta_{-1}=D_{-1}\circ f$.
\par A pair of symmetric theta structures $f_1$ for $\mathcal{L}$ and $f_2$ for $\mathcal{L}^2$ will be called \emph{compatible} if
$$f_2\circ \varepsilon_2=E_2\circ f_1$$
$$f_1 \circ \eta_2=H_2\circ f_2 $$
\end{definition}
We will now define a particular kind of theta structure on a principally polarized abelian variety. Let $(A, \lambda)$ be a p.p.a.v. Recall that this means that $\lambda: A \stackrel{\sim}{\rightarrow} A^\vee$ is an isomorphism satisfying $\lambda=\lambda^\vee$ such that the pullback of the Poincar\'e line bundle along the composition $A\stackrel{\Delta}{\rightarrow} A\times A \stackrel{(\id_{A}, \lambda)}{\longrightarrow} A\times A^\vee$ is ample. Denote this pullback by $\mathcal{L}_2$. It is well-known that the line bundle $\mathcal{L}_2$ is symmetric and induces $2 \lambda$.
\begin{definition}
A principally polarized abelian variety with a $\Gamma(4,8)$-level structure is a tuple $(A,\lambda, f_2)$ where $(A, \lambda)$ is a principally polarized abelian variety over a field $k$ with $\cha(k)\neq 2$ and $f_2$ is a symmetric theta structure for $\mathcal{L}_2^2$.
\end{definition}
Since $K(\mathcal{L}_2^2)=A[4]$, a $\Gamma(4,8)$-level structure induces an isomorphism $A[4]\cong \underline{\mathbb{Z}/4\mathbb{Z}^g} \oplus \mu_4^g$.
\par Recall that every principal polarization is induced from a symmetric degree $1$ line bundle. This line bundle is not unique, however. But one advantage of a $\Gamma(4,8)$-level structure is that it fixes a canonical choice $\mathcal{L}$. It is characterized by putting the quadratic form $e_*^\mathcal{L}(\cdot)$ into normal form:
\begin{lem}\label{Gamma48Can} 
Let $(A,\lambda, f_2)$ be a p.p.a.v. with a $\Gamma(4,8)$-level structure. Then there is a unique symmetric line bundle $\mathcal{L}$ such that:
\begin{itemize}
\item[i)] $\mathcal{L}^2\cong \mathcal{L}_2$.
\item[ii)] The quadratic form $e_*^\mathcal{L}: A[2](k)\rightarrow \{\pm 1\}$ is given by the matrix
$$\begin{pmatrix}
0 & I_g\\
0 & 0
\end{pmatrix}$$
with respect to the basis of $A[2](k)$ induced by $f_2$.
\end{itemize}
\end{lem}
\begin{proof}
By \cite[p. 214]{MumAV} there is a line bundle $\mathcal{L}$ such that $\mathcal{L}^2\cong \mathcal{L}_2$, at least after an extension of the ground field. We will see in the proof of Theorem \ref{Mult2} that this field extension is not necessary.
\par It is elementary to verify that a line bundle whose square is symmetric, must be symmetric. Therefore $\mathcal{L}$ is symmetric. However, $\mathcal{L}$ is only unique up to translating by a $2$-torsion point.
\par On the other hand notice that the $\vartheta$-structure $f_2$ induces an isomorphism $A[4]\cong \underline{\mathbb{Z}/4\mathbb{Z}^g} \oplus \mu_4^g$. This implies that $A[2]$ is constant and we obtain a basis of $A[2](k)$. The quadratic form $e_*^\mathcal{L}(\cdot)$ is given by an upper triangular matrix $X\in \Mat_{2g,2g}(\mathbb{F}_2)$ with respect to our basis of $A[2](k)$. It is easy to check from the definitions that this basis is symplectic. Therefore Proposition \ref{QFProp} ii) implies 
$$X+X^t=\begin{pmatrix}
0 & I_g\\
-I_g & 0
\end{pmatrix}=\begin{pmatrix}
0 & I_g\\
I_g & 0
\end{pmatrix}$$
Thus $X$ differs from $\begin{pmatrix}
0 & I_g\\
0 & 0
\end{pmatrix}$ only on the diagonal. An elementary calculation shows that we can achieve $X=\begin{pmatrix}
0 & I_g\\
0 & 0
\end{pmatrix}$ by translating with a unique point $P\in A[2](k)$. Therefore $t_P^*(\mathcal{L})$ has all the desired properties.
\end{proof}
\subsection{Representations of a theta group}
Let $A$ be an abelian variety and $\mathcal{L}$ an ample line bundle.
\begin{definition}
We define a representation of $\tG{}$ on $\HH^0(A, \mathcal{L})$ as follows:
\par For a $k$-algebra $R$ and a point $(x, \phi)\in \tG{}(R)$ we define
$$U_{(x, \phi)}: \HH^0(A, \mathcal{L})\otimes_k R \rightarrow \HH^0(A, \mathcal{L})\otimes_k R$$
to be the composition
$$\HH^0(A, \mathcal{L})\otimes_k R=\HH^0(A_R, \mathcal{L}_R)\stackrel{\phi}{\rightarrow} \HH^0(A_R, t_x^*\mathcal{L}_R)\stackrel{t_{-x}^*}{\rightarrow} \HH^0(A_R, \mathcal{L}_R)=\HH^0(A, \mathcal{L})\otimes_k R$$
\end{definition}
\begin{definition}\label{DefVdelta} We define a representation of $\mathcal{G}(\delta)$ as follows: The underlying vector space is $V(\delta)=k^{Z_\delta}$, the vector space of maps of sets $f: Z_\delta\rightarrow k$. We define an action of $\mathcal{G}(\delta)$ as follows:
$$(U_{r,x,\chi} f)(y))=r \cdot \chi(y) \cdot f(x+y) $$
\end{definition}
Recall the following theorem:
\begin{thm}
\begin{enumerate}
\item[i)] The representation of $\tG{}$ on $\HH^0(A, \mathcal{L})$ is absolutely irreducible and of weight 1.
\item[ii)] The representation of $\mathcal{G}(\delta)$ on $V(\delta)$ is absolutely irreducible and of weight 1.
\item[iii)] There is a unique absolutely irreducible weight 1 representation of $\tG{}$.
\end{enumerate}
\end{thm}
\begin{proof}
For i) see \cite[p. 710]{Sekiguchi}. The assertion ii) is proven in \cite[p. 295]{EqDefAV}. iii) is proven in \cite[theorem 8.32]{vdGMooAV} in the case where $k$ is algebraically closed. The general case follows with a descent argument, see \cite[Lemma 3.17]{AP}
\end{proof}
We will now define theta nullpoints. Assume that $p\nmid \deg(\mathcal{L})$ and $\tG{}$ admits a theta structure of type $\delta$. By \cite[p. 294]{EqDefAV} the latter condition follows from the former, at least after a finite separable field extension. This theta structure induces an isomorphism $\HH^0(A, \mathcal{L})\cong V(\delta)$ unique up to a scalar. Choose a trivialization $\mathcal{L}_{|0}\cong k$. This determines an evaluation at $0$ map $\ev_0: \HH^0(A, \mathcal{L}) \rightarrow k$
\begin{definition}
The \emph{theta nullpoint} $q_\mathcal{L}: Z_\delta \rightarrow k$ is defined to be the coordinate vector of the composition $V(\delta)\cong \HH^0(A, \mathcal{L})\stackrel{\ev_0}{\rightarrow} k$.
\end{definition}
Notice that the theta nullpoint is only defined up to a scalar. Furthermore, it will depend on the choice of the theta structure.
\subsection{Theta Nullvalues of principally polarized Abelian Varieties}
Throughout this section $(A, \lambda, f_2)$ will be a p. p. a. v. with a $\Gamma(4,8)$ level structure. By Lemma \ref{Gamma48Can} there is a canonical symmetric line bundle $\mathcal{L}$ inducing $\lambda$. We will choose $s\in \HH^0(A, \mathcal{L})$, a non-zero element. Then $s$ is unique up to a scalar. We will denote by $\Theta=\mathcal{V}(s)$ the divisor of vanishing of $s$, i.e. a \emph{theta divisor}.
\par Choose a trivialization $\mathcal{L}_{|0}\cong k$. This induces an evaluation map $\ev_0:\HH^0(A, \mathcal{L}^4)\rightarrow k$. Via $\ev_0$ and the theta structure $f_2$ on $\mathcal{L}^4$ we obtain the theta nullvalues $q_{\mathcal{L}^4}:Z_4\rightarrow k$ (well-defined up to a scalar). Frequently instead of $q$ one uses a partial Fourier transform of it:
\begin{definition} In the situation above choose a square root of $-1$, say $i$ in some field extension $K$ of $k$. For $a, b \in \ZZ^g$ we define
$$\vartheta_{a,b}=\sum_{c\in Z_4, c \equiv a (2)} i^{c^t b} q_{\mathcal{L}^4}(c)$$
\end{definition}
\begin{lem}
$\vartheta_{a,b}$ is in $k$ and independent of the choice of $i$.
\par Furthermore, $\vartheta_{a,b}=0$ if $a^t b\equiv 1 (2)$.
\end{lem}
\begin{proof}
We show the first statement. Indeed, assume first $a^t b \equiv 0 (2)$. Then the exponent $c^t b$ in the definition of $\vartheta_{a,b}$ is always even. Therefore $i^{c^t b}=(-1)^{\frac{c^t b}{2}}$ is independent of the choice of $i$. Furthermore, $\vartheta_{a,b}$ is in $k$.
\par Assume now $a^t b \equiv 1 (2)$. Clearly it suffices to show $\vartheta_{a,b}=0$. Indeed, by Mumford's inversion formula \cite[p. 331]{EqDefAV} we have $q_{\mathcal{L}^4}(-x)=q_{\mathcal{L}^4}(x)$ for all $x\in Z_4$. Therefore
$$\vartheta_{a,b}=\sum_{c\in Z_4, c \equiv a (2)} i^{c^t b} q_{\mathcal{L}^4}(c)=\sum_{c\in Z_4, -c \equiv a (2)} i^{-c^t b} q_{\mathcal{L}^4}(-c)=$$
$$\sum_{c\in Z_4, c \equiv a (2)} (-i)^{c^t b} q_{\mathcal{L}^4}(c)=-\vartheta_{a,b}$$
because the exponent $c^t b$ is always odd.
\end{proof}
We will give another description of $\vartheta_{a,b}$ which will be useful later for their computation. This relies on the following duplication formula:
\begin{prop}\label{Mult2} The image of
$$[2]^*: \HH^0(A, \mathcal{L})\rightarrow \HH^0(A, \mathcal{L}^4)\cong V(4)$$
is a one-dimensional vector space generated by
$$\sum_{x\in 2 Z_4} e_x$$
\end{prop}
\begin{proof}
This does not follow from Mumford's theta duplication formula \cite[p. 330]{EqDefAV} since $\mathcal{L}$ is not totally symmetric. The argument given below easily generalizes to an extension of Mumford's formula to all symmetric line bundles, but for sake of simplicity we shall confine ourselves to the special case of a symmetric line bundle of degree $1$. A related generalization of Mumford's duplication formula follows from results of Kempf \cite[\S 3, Theorem 8]{Kempf}.
\par Indeed, consider the isogeny $[2]: A \rightarrow A$. Choose an isomorphism $\psi: \mathcal{L}^4 \stackrel{\sim}{\rightarrow} [2]^* \mathcal{L}$. The properties of theta groups with respect to the isogeny $[2]^*$ give a level subgroup $\tilde{K}\subset \mathcal{G}(\mathcal{L}^4)$ above $A[2]$. Since $A[2]$ is a constant group scheme, it suffices to consider only $k$-valued points. The elements of $\tilde{K}(k)$ have the following description: For a point $x\in A[2](k)$ denote by $\phi_x: \mathcal{L}^4 \rightarrow t_x^*\mathcal{L}^4$ the composition $\mathcal{L}^4\stackrel{\psi}{\rightarrow} [2]^* \mathcal{L}=t_x^* ([2]^* \mathcal{L}) \stackrel{(t_x^*\psi)^{-1}}{\rightarrow} t_x^* \mathcal{L}^4$. Then clearly $\phi_x$ is independent of the choice of the isomorphism $[2]^*\mathcal{L} \cong \mathcal{L}^4$. Furthermore, $(x, \phi_x)\in \tG{4}(k)$ is the element in $\tilde{K}(k)$ above $x$. This description of $\tilde{K}$ implies that the line bundle $\mathcal{L}$ is defined over $k$ as claimed in Lemma \ref{Gamma48Can}. We have a commutative diagram
$$ \begin{tikzcd}
\mathcal{L}^{4} \arrow["\phi_x",r]\arrow["\psi",dd]& t_x^*\mathcal{L}^{4}\arrow["t_x^* \psi", d]\\
 & t_x^*(\left[2\right]^* \mathcal{L}) \arrow[equal]{d}\\
\left[2\right]^*\mathcal{L} \arrow{r}[swap]{\left[2 \right]^* \id_{\mathcal{L}}}& \left[ 2 \right]^* t_{0}^* \mathcal{L}
\end{tikzcd}$$
Thus by Definition \ref{Defeta} we see $\varepsilon_2 (\ker(\eta_2)) \subseteq \tilde{K}$ where
$$\mathcal{G}(\mathcal{L}^4)\stackrel{\varepsilon_2}{\leftarrow} \mathcal{G}(\mathcal{L}^2)\stackrel{\eta_2}{\rightarrow} \mathcal{G}(\mathcal{L})$$
\par We will now use the inclusion $\varepsilon_2 (\ker(\eta_2)) \subseteq \tilde{K}$ to construct certain elements of $\tilde{K}(k)$. This will in fact imply that the inclusion is an equality. Indeed, by \cite[p. 317]{EqDefAV} there is a unique $\vartheta$-structure on $\mathcal{L}^2$ such that $(f_1,f_2)$ are compatible. We obtain a commutative diagram
$$\begin{tikzcd}
\tG{4} \arrow["f_2",d] & \tG{2} \arrow["\eta_2",r]\arrow["\varepsilon_2", l]\arrow["f_1",d] & \tG{}\\
\mathcal{G}(4) & \mathcal{G}(2) \arrow["E_2",l]\\
\end{tikzcd}$$
Consider for $a\in Z_2$ the element $z_a=f_1^{-1}(1,a,0)\in \tG{2}(k)$. Then $z_a$ is an element of order $2$ and thus by Proposition \ref{etae*} we have $\eta_2(z_a)=e_*^\mathcal{L}(P_a)$ where $P_a$ is the image of $z_a$ in $K(\mathcal{L}^2)=A[2](k)$. Now since we assumed that $e_*^\mathcal{L}$ is in normal form, an easy calculation yields $e_*^\mathcal{L}(P_a)=1$. Therefore $z_a\in \ker(\eta_2)$. This gives the element $\varepsilon_2(z_a)= f_2^{-1}(1,[2] a,0)\in \tilde{K}(k)$. Similarly consider for $b\in Z_2$ the corresponding character $\chi_b: Z_2 \rightarrow k^\times,\, x\mapsto (-1)^{x^t b}$. Then $z_b=f_1^{-1}(1,0,\chi_b)\in \tG{2}(k)$ gives via the same argument the element $\varepsilon_2(z_b)= f_2^{-1}(1,0,r^D(\chi_b))\in \tilde{K}(k)$. The elements $\varepsilon_2(z_a),\, \varepsilon_2(z_b)$ generate a subgroup of $\tG{4}(k)$ isomorphic to $Z_2 \oplus Z_2$. Therefore $\varepsilon_2 (\ker(\eta_2)) = \tilde{K}$.
\par To compute now the image of
$$[2]^*: \HH^0(A, \mathcal{L})\rightarrow \HH^0(A, \mathcal{L}^4)\, ,$$
we can invoke Mumford's isogeny formula \cite[\S 1, Theorem 4]{EqDefAV}. The proposition follows.
\end{proof}
Choose a primitive 4th root of unity in some field extension $K$ of $k$. This determines an isomorphism $\mu_4 \cong \underline{\mathbb{Z}/4\mathbb{Z}}$ over $K$ and thus an isomorphism $\underline{Z_4}^D\cong \underline{Z_4}$. Therefore the theta structure $f_2$ induces a bijection of sets $f_{2,K}^{-1}:  K^\times \times Z_4 \times Z_4 \stackrel{\sim}{\rightarrow} \mathcal{G}(\mathcal{L}^4)(K)$. We have the following description of $\vartheta_{a,b}$: 
\begin{thm}\label{ThetaNullDescr}
For all $a,b\in \mathbb{Z}^g$
$$\ev_0( U_{f_{2,K}^{-1}(1,a,b)}  ([2]^* s) )$$
is in $k$ and is independent of the choice of a primitive 4th root of unity.
\par Furthermore, there is a $\lambda\in k^\times$ such that for all $a,b\in \mathbb{Z}^g$ we have
$$\vartheta_{a,b}=\lambda \cdot \ev_0( U_{f_{2,K}^{-1}(1,a,b)}  ([2]^* s) )$$
\end{thm}
\begin{proof}
Follows from Proposition \ref{Mult2} and Definition \ref{DefVdelta} via an elementary computation.
\end{proof}
\begin{cor}\label{ThetaVan} For $a, b \in \mathbb{Z}^g$ denote by $P_{a,b}\in A[2](k)$ the 2-torsion point with coordinate vector $\begin{pmatrix}
a & b
\end{pmatrix}$ in the basis of $A[2](k)$ induced by $f_1$. Then $\vartheta_{a,b}=0$ if and only if $P_{a,b}\in \Theta$.
\end{cor}
\begin{proof}
By definition $\Theta=\mathcal{V}(s)=\{x\in A| s_{|x}=0\}$. Then $\mathcal{V}([2]^*(s))=[2]^{-1}(\Theta)$. On the other hand let us denote by $P'_{a,b}\in K(\mathcal{L}^4)(K)=A[4](K)$ the image of $f_{2,K }^{-1}(1,a,b)$. Then we have $\mathcal{V}(U_{f_{2,K}^{-1}(1,a,b)}  ([2]^* s))=\mathcal{V}(t_{P'_{a,b}}^* ([2]^* s))$ because the divisor of vanishing does not change after applying an isomorphism of line bundles. Thus 
$$\mathcal{V}(U_{f_{2,K}^{-1}(1,a,b)}  ([2]^* s))=t_{P'_{a,b}}^{-1}( [2]^{-1}( \Theta))=[2]^{-1}(t_{2P'_{a,b}}(\Theta)))=[2]^{-1}(t_{P_{a,b}}^{-1}(\Theta)))$$
Now $\vartheta_{a,b}=0$ iff $\ev_0( U_{f_{2,K}^{-1}(1,a,b)}  ([2]^* s) )=0$ iff $0\in \mathcal{V}(U_{f_{2,K}^{-1}(1,a,b)}  ([2]^* s))$ iff $P_{a,b}=0$. This proves the corollary.
\end{proof}
\begin{rem}
Notice that in all the results in this section it was important that the quadratic form $e_*^\mathcal{L}$ was put into normal form. Otherwise one should replace the expression $c^t b$ in the definition of $\vartheta_{a,b}$ suitably.
\end{rem}
\par As we will see below, some formulae, e.g. the formulae for hyperelliptic curves, will only depend on the squares $\vartheta_{a,b}^2$. But to compute these squares we do not necessarily need to know the theta nullvalues $q_{\mathcal{L}^4}$. In fact it suffices to know $q_{\mathcal{L}^2}$. This follows from the following lemma:
\begin{lem}
After suitably normalizing the isomorphisms $\HH^0(A, \mathcal{L}^2)\cong V(2),\, \HH^0(A, \mathcal{L}^4)\cong V(4)$ we have for all $a,b \in \ZZ^g$:
$$\vartheta_{a,b}^2=\sum_{x\in Z_2} (-1)^{x^t b} q_{\mathcal{L}^2}(x)q_{\mathcal{L}^2}(x+a)$$
\end{lem}
\begin{proof}
Since $\mathcal{L}^2$ is totally symmetric, we can apply Mumford's addition formula \cite[formula (A) on p. 332]{EqDefAV} to $\mathcal{L}^2$. Thus for all $x, y\in Z_2$
$$q_{\mathcal{L}^2}(x)q_{\mathcal{L}^2}(y)=\sum_{\substack{u,v\in Z_4\\ u+v=[2]x\\ u-v=[2]y}} q_{\mathcal{L}^4}(u) q_{\mathcal{L}^4}(v)$$
Notice that the formula looks slightly different than Mumford's. He has a different notation for the embedding $[2]: Z_2 \hookrightarrow Z_4$. Choose, as before, a square root of $-1$ in some field extension denoted $i$. Then
$$\sum_{x\in Z_2} (-1)^{x^t b} q_{\mathcal{L}^2}(x)q_{\mathcal{L}^2}(x+a)= \sum_{x\in Z_2} \sum_{\substack{u,v\in Z_4\\ u+v=[2]x\\ u-v=[2](x+a)}} (-1)^{x^t b} q_{\mathcal{L}^4}(u) q_{\mathcal{L}^4}(v)=$$
$$\sum_{x\in Z_2} \sum_{\substack{u,v\in Z_4\\ u+v=[2]x\\ u-v=[2](x+a)}} i^{(u+v)^t b} q_{\mathcal{L}^4}(u) q_{\mathcal{L}^4}(v)=\sum_{\substack{u,v\in Z_4\\ u\equiv v \equiv a (2)}} i^{(u+v)^t b} q_{\mathcal{L}^4}(u) q_{\mathcal{L}^4}(v)=\vartheta_{a,b}^2$$
\end{proof}
\subsection{Complex analytic comparison theorem}
Over the field $\mathbb{C}$ Mumford's algebraic theta nullvalues agree with the analytic theta nullvalues defined via Riemann theta functions with rational characteristics. Recall their definition:
\begin{definition} Denote by $\mathfrak{H}_g$ the Siegel upper half space. For $a,b\in \QQ^g$ we define the theta function with rational characteristics
$$\thetacha{a}{b}: \mathbb{C}^g\times \mathfrak{H}_g\rightarrow \mathbb{C} $$
$$\thetacha{a}{b}  (z, \Omega)= \sum_{n\in {\ZZ}^g} \exp\left( \pi i (n+a)^t \Omega (n+a) +2 \pi i (n+a)^t(z+b)\right) $$
\end{definition}
The following theorem gives the connection between algebraic and analytic theta nullvalues. Strictly speaking it is not used in this paper. However, we believe that recalling it will help to clarify the general picture.
\begin{thm} \label{ComplxCo}
Let $(A,  \lambda, f_2)$ be a principally polarized abelian variety over $\CC$ with a $\Gamma(4,8)$-level structure. Then there is a symplectic basis of $H_1(A, \ZZ)$ unique up to $\Gamma(4,8)$ such that the associated period matrix $\Omega$ satisfies: There is an $\lambda\in \CC^\times$ such that for all $a,b \in \mathbb{Z}^g$
$$\vartheta_{a,b}=\lambda \cdot  \thetacha{a/2}{b/2}(0, \Omega)$$
\end{thm}
\begin{proof}
Follows from \cite[Proposition 5.11]{TataIII}.
\end{proof}
\section{Finite commutative group schemes}\label{Chapfcgrschem}
\subsection{Truncated Witt group schemes}
Let $k$ be a field of characteristic $p>0$. Recall that the ring of ($p$-typical) Witt vectors of a ring $R$ is $W(R)=R^\mathbb{N}$ with addition and multiplication given by Witt polynomials. Because the ring structure is given by polynomials the functor $k-\text{algebras}\rightarrow \text{Rings},\, R\mapsto W(R)$ is represented by an affine ring scheme over $k$ with underlying scheme $\Spec(k[x_0, x_1, \ldots])$. We denote this ring scheme by $W$. Similarly the truncated Witt rings define a ring scheme denoted $W_n$
\begin{definition}
Let $m,n$ be positive integers. We define $W_{m,n}=\ker(F^m: W_n \rightarrow W_n)$ where $F: (a_0, \ldots, a_{n-1})\mapsto (a_0^p, \ldots, a_{n-1}^p)$ is the relative Frobenius map.
\end{definition}
Because $F$ is an additive homomorphism, $W_{m,n}$ is a group scheme. It is finite and connected with Hopf algebra $k[x_0, \ldots, x_{n-1}]/(x_0^{p^m}, \ldots x_{n-1}^{p^m})$.
\subsection{Duality of truncated Witt group schemes}
In this section we recall that the Cartier dual of $W_{m,n}$ is $W_{n,m}$ where the duality is given by the Artin-Hasse exponential.
\begin{definition}
The Artin-Hasse exponential series is defined to be the series
$$\exp_\text{AH}(t)=\exp\left(-\sum_{i=0}^\infty \frac{t^{p^i}}{p^i} \right) \in \mathbb{Q}\llbracket t \rrbracket$$
\end{definition}
By \cite[p. 52]{Demazure} one has $\exp_\text{AH}\in \mathbb{Z}_{(p)}\llbracket t \rrbracket$, i.e. the denominators of the coefficients are not divisible $p$.
\begin{definition}
Let $R$ be $k$-algebra. Consider
$$\mathbb{W}(R)=\{(a_0, a_1,\ldots)\in W(A)| a_i=0 \text{ for all but finitely many } i\}\,.$$
We define a map
$$E_\text{AH}: \mathbb{W}(R)\rightarrow R$$
$$(a_0, a_1, \ldots ) \mapsto \prod_{i=0}^\infty \exp_\text{AH}(a_i)$$
\end{definition}
It is clear that the product in the definition above is actually finite and thus $E_\text{AH}$ is well-defined.
\par We define the map of schemes $\sigma_n: W_n \rightarrow W, \, (a_0, \ldots, a_{n-1})\mapsto (a_0, \ldots, a_{n-1}, 0, \ldots)$. N. B. $\sigma_n$ is not a morphism of group schemes. The duality theory of truncated Witt group schemes is given by the following theorem:
\begin{thm}\label{WittDual} For $x\in W_{m,n}(R), y\in W_{n,m}(R)$ define
$$\langle x,y \rangle=E_\text{AH}(\sigma_n(x) \cdot_W \sigma_m(y))$$
Then $\langle \cdot , \cdot \rangle: W_{m,n}\times W_{n,m}\rightarrow \mathbb{G}_m$ is biadditive and induces an isomorphism
$$W_{m,n}\cong (W_{n,m})^D$$
Furthermore,
$$\langle F(x), y\rangle= \langle x, t(y)\rangle$$
where $t: W_{m,n}\rightarrow W_{m,n},\, (x_0, \ldots, x_{n-1})\mapsto (0, x_0, \ldots, x_{n-2})$ is induced by Verschiebung on the Witt ring.
\end{thm}
\begin{proof}
See \cite[p. 61]{Demazure}.
\end{proof}
\begin{cor} The Verschiebung map $V: W_{m,n}\rightarrow W_{m,n}$ is given by $t$.
\end{cor}
\begin{proof}
Follows from the theorem because $V: W_{m,n}\rightarrow W_{m,n}$ is defined to be the Cartier dual of $F: W_{m,n}^D\rightarrow W_{m,n}^D$ (see \cite[p. 28]{Demazure}).
\end{proof}
\subsection{Dieudonn\'e theory}
We will collect here the results from Dieudonn\'e theory we shall need. We will make use of covariant Dieudonn\'e theory. The reason for this choice is that we want to work with surjections from truncated Witt group schemes and thus using the covariant theory is more natural. Let $k$ be a perfect field of characterstic $p>0$. Recall that the Dieudonn\'e ring $D(k)$ is defined to be the non-commutative ring $W(k)[F,V]$ with the relations $Fa=\sigma(a) F, V \sigma(a)=a V,\, \forall a\in W(k),\, FV=VF=p$. We also define $D_{m,n}=\frac{D(k)}{D(k) V^m+D(k) F^n}$. Since $D(k) V^m+D(k) F^n$ is a both-sided ideal, $D_{m,n}$ is a ring. Now $D_{m,n}$ will turn out to be the opposite algebra of the endomorphism algebra of $W_{m,n}$ via the following map:
\begin{definition}
We define a map $D_{m,n}\rightarrow \End(W_{m,n})^\text{op}$ as follows: For $a\in W(k)$
$$a\mapsto (W_{m,n}\rightarrow W_{m,n},\, x\mapsto a\cdot_W x)$$
$$F\mapsto V$$
$$V\mapsto F$$
Notice the switch of $F$ and $V$.
\end{definition}
Now we define the Dieudonn\'e module of a local-local group scheme killed by $F^m, V^n$ as follows:
\begin{definition}
Let $G$ be a finite commutative group scheme over $k$ such that $F^m, V^n$ are zero on $G$. Define
$$M_{m,n}(G)=\Hom_k(W_{m,n}, G) $$
Then $M_{m,n}(G)$ becomes a $D_{m,n}$-module via the map $D_{m,n}\rightarrow \End(W_{m,n})^\text{op}$.
\end{definition}
\begin{thm}
The functor $$M_{m,n}:\left\{
\begin{tabular}{@{}l@{}}
 finite commutative\\
  group schemes\\
of killed by $F^m, V^n$
\end{tabular}\right\}\longrightarrow  \left\{
\begin{tabular}{@{}l@{}}
  finitely generated left $D(k)$-modules killed\\
  by $V^m$ and\\ by $F^n$ 
\end{tabular}\right\}$$
is an exact equivalence. Furthermore, one has $M_{m,n}(W_{m,n})=D_{m,n}$.
\end{thm}
\begin{proof}
Usually Dieudonn\'e theory is stated via a contravariant functor, see e.g. \cite[p. 65]{Demazure}. The covariant version follows automatically from the contravariant one. Indeed, the existence of a contravariant equivalence and the description of the contravariant Dieudonn\'e module of $W_{m,n}$ of \cite[p. 66]{Demazure} implies our theorem.
\end{proof}
\begin{rem}
Notice that our definition of the $W(k)$-module structure of the Dieudonn\'e module differs from the usual one by a twist by a power of Frobenius. Also beware that $M_{m,n}$ depends on $m,n$ and changing it will also cause a twist by Frobenius.
\par However, our convention is easier to use for explicit computations as we do not have twists by powers of the inverse of Frobenius in the description of endomorphisms of truncated Witt group schemes. This is also crucial for the computation of families because the inverse of Frobenius is not defined on function fields.
\end{rem}
\subsection{Representation theory of finite commutative group schemes}\label{Repgrsch}
We shall make use of the representation theory of finite commutative group schemes. Throughout this section $G$ will be a finite commutative group scheme over a field $k$ and $R_G$ denotes its Hopf algebra, $c: R_G\rightarrow R_G\otimes R_G$ denotes the comultiplication, and $u: R_G\rightarrow k$ denotes the unit. Recall the following definition from \cite[p. 22]{Water}
\begin{definition} An $R_G$-comodule is a $k$-vector space $V$ together with a $k$-linear map $c_V: V\rightarrow V \otimes_k R_G$ such that
$$\begin{tikzcd}
V \arrow{r}{c_V}\arrow{d}{c_V} & V \otimes R_G \arrow{d}{\id \otimes c}\\
V\otimes R_G \arrow{r}{c_V\otimes \id}& V \otimes R_G \otimes R_G
\end{tikzcd}$$
and 
$$\begin{tikzcd}
V \arrow{r}{c_V}\arrow[equal]{d} & V\otimes R_G\arrow{d}{\id\otimes u} \\
V \arrow{r}{\sim} &V\otimes k
\end{tikzcd}$$
commute.
\end{definition}
Let $G^D$ denote the Cartier dual of $G$. Cartier duality can be used to understand the representation theory of $G$ via the following proposition:
\begin{prop} The functor $M\mapsto M^\vee$ induces an anti-equivalence of categories
$$\{R_G\text{-comodules of finite dimension} \} \rightarrow \{R_{G^D}\text{-modules of finite dimension} \}$$
\end{prop}
\begin{proof}
Follows immediately from the definitions.
\end{proof}
\begin{cor}\label{WittReps}
A representation $\varphi: W_{n,m}\rightarrow \Gl(M)$ is of the following form:
There exist unique $X_0,\ldots, X_{n-1}\in \End(M)$ which commute and $X_\nu^{p^m}=0, \forall \nu$ such that for any $k$-scheme $T$, any $a=(a_0, \ldots, a_{m-1})\in W_{n,m}(T)$ and any vector $v\in \HH^0(T, M\otimes \mathcal{O}_T)$
$$ \varphi(a)(v)=E_{\text{AH}}((X_0, \ldots, X_{n-1},0,\ldots)\cdot_W a )(v)$$
\end{cor}
\begin{proof}
Follows from the proposition and Theorem \ref{WittDual}.
\end{proof}
Any representation of $W_{n,m}$ has a non-zero invariant vector because $W_{n,m}$ is unipotent. We shall make this result explicit by using the previous corollary for $m=1$ or $n=1$:
\begin{cor}\label{WittInvV} Consider $R=R_{W_{1,n}}=k[x_0]/(x_0^{p^n})$. Let $M$ be a finite dimensional $R$-comodule with comodule map $c_M:M\rightarrow M\otimes R$. Let $v\in M\setminus \{0\}$ be any vector. Write
$$c_M(v)=\sum_{i=0}^{p^n-1} v_i \otimes x_0^i$$
Suppose $j$ is the largest $i$ such that $v_i\neq 0$. Then $v_j\in M^{W_{1,n}}$.
\end{cor}
\begin{proof}
Clearly an $i$ with $v_i\neq 0$ exists because $v_0=v\neq 0$. Suppose $j$ is the largest such $i$. We need to prove that $v_j\in M^{W_{1,n}}$. Indeed, by Corollary \ref{WittReps} the comodule $M$ can be described with $X_0,\ldots, X_{n-1}\in \End(M)$ which commute and $X_\nu^{p}=0$. Then we have
$$c_M(v)=E((X_0, \ldots, X_{n-1},0, \ldots)\cdot_W (x_0,0,\ldots) )(v)=$$
$$E((x_0 X_0, x_0^p X_1, \ldots x_0^{p^{n-1}} X_{n-1}))(v)=\left(\prod_{\nu=0}^{n-1} \exp_{\text{AH}}(x_0^{p^\nu} X_\nu)\right)(v)$$
Now since $X_\nu^{p}=0$ for all $\nu$, the Artin-Hasse exponential series collapses after the $p$th term and thus coincides with the usual exponential series. Thus
$$c_M(v)= \left(\prod_{\nu=0}^{n-1} (\sum_{i=0}^{p-1} \frac{(x_0^{p^\nu} X_\nu)^i}{i!})\right)(v)= \sum_{\underline{i}\in \{0,\ldots, p-1\}^m } \frac{x_0^{\sum_{\nu=0}^{n-1} i_\nu p^{\nu}} \underline{X}^{\underline{i}}(v) }{\underline{i}!}$$
where we use the shorthand notation $\underline{X}^{\underline{i}}=\prod_{\nu=0}^{n-1} X_\nu^{i_\nu}$ and $\underline{i}!=\prod_{\nu=0}^{n-1} i_\nu !$. This implies that $v_{j+p^\nu}$ is equal to $X_\nu(v)$ up to a non-zero factor for any $\nu=0,\ldots, n-1$. But by our assumption on $j$ one has $v_{j+p^\nu}=0$. Therefore $v\in \bigcap_{\nu=0}^{n-1} \ker(X_\nu)$. On the other hand $\bigcap_{\nu=0}^{n-1} \ker(X_\nu) = M^{W_{1,n}}$. This proves the claim.

\end{proof}
\begin{cor}\label{WittInvF} Consider $R=R_{W_{m,1}}=k[x_0, \ldots, x_{m-1} ]/(x_0^p,\ldots, x_{m-1}^p)$.
Let $M$ be a finite dimensional $R$-comodule with comodule map $c_M:M\rightarrow M\otimes R$. Let $v\in M\setminus \{0\}$ be any vector. Write
$$c_M(v)=\sum_{\underline{i}\in \{0,\ldots, p-1\}^m } v_{\underline{i}} \otimes \underline{x}^{\underline{i}}$$
Here we will use the shorthand notation $\underline{x}^{\underline{i}}=\prod_{\nu=0}^{m-1} x_\nu^{i_\nu}$ for a tuple $\underline{i}=(i_0, \ldots, i_{m-1})$. We define a total ordering on $\{0,\ldots, p-1\}^m$ by saying $\underline{i}\succeq \underline{i'}$ if $\sum_{\nu=0}^{m-1} i_\nu p^\nu \geqslant \sum_{\nu=0}^{m-1} i'_\nu p^\nu$.
\par Suppose $\underline{j}$ is maximal among the $\underline{i}$ satisfying $v_{\underline{i}}\neq 0$. Then $v_{\underline{j}}\in M^{W_{m,1}}$.
\end{cor}
\begin{proof}
The proof works the same way as in Corollary \ref{WittInvV}.
\end{proof}
We will now explain how Corollary \ref{WittInvV} (resp. Corollary \ref{WittInvF}) leads to an algorithm to compute a non-zero invariant vector of a $W_{m,n}$-module. The basic idea is to inductively construct vectors invariant under increasing subgroup schemes. The filtration will be obtained from the powers of $V$ (resp. $F$).
\par There is no significant difference between the two methods. We prefer using powers of $V$ because it is easier to implement. However, the other option has a slight advantage concerning speed. But this only pays off if $g$ is very large and there was no difference in our experiments.
\par Let us start to discuss the method using powers of $V$. The other variant is analogous and will be omitted. Indeed, let $M$  be a finite dimensional $W_{m,n}$-comodule. There is a filtration $0=W_{0,n}\hookrightarrow W_{1,n} \hookrightarrow \ldots \hookrightarrow W_{m,n}$ where the inclusions are induced by the Verschiebung map. The Hopf algebra of $W_{m,n}$ is given by $R=R_{W_{m,n}}=k[x_0, \ldots, x_{m-1}]/(x_0^{p^n}, \ldots, x_{m-1}^{p^n})$. In this description the closed immersion $W_{i,n} \hookrightarrow W_{m,n}$ corresponds to the ideal $(x_0, \ldots, x_{m-i-1})$.
\par We will inductively construct a non-zero vector invariant under $W_{i,n}$. Indeed, let $v\in M\backslash \{0\}$ be any vector. Then $v$ is invariant under the trivial group. Suppose now we have already constructed a $v_{i-1}\in M^{W_{i-1,n}}\backslash \{0\}$. Let us apply the comodule map $c_M: M \rightarrow M \otimes R$ to $v_{i-1}$. This leads to the following lemma:
\begin{lem} In the notation from above, let $R_i$ denote the subalgebra of $R$ generated by $(x_0, \ldots, x_{m-i})$. Then $c_M(v_{i-1})\in M \otimes R_i$.
\par Furthermore, let us expand $c_M(v_{i-1})$ in the monomial basis of $R_i$. Suppose $v_i \otimes x_{m-i}^j$ is the non-zero tensor in this expansion involving only $x_{m-i}$ and having $j$ maximal with respect to this property. Then $v_i\in M^{W_{i,n}}$.
\end{lem}
\begin{proof}
By assumption we have $v_{i-1}\in M^{W_{i-1,n}}$. Therefore the action of $W_{m,n}$ on $v_{i-1}$ factors through the quotient $W_{m,n}/W_{i-1,n}$. But the quotient map $W_{m,n}\rightarrow W_{m,n}/W_{i-1,n}$ is induced from the inclusion $R_i \subset R$. This proves the first claim.
\par Next we need to construct a vector invariant under $W_{i, n}$. The action of the subgroup scheme $W_{i,n}\hookrightarrow W_{m,n}$ factors through the quotient $W_{i,n}/W_{i-1,n} \cong W_{1,n}$. Therefore we obtain an invariant vector $v_i$ via the construction of Corollary \ref{WittInvV}. It is easy to see that the vector $v_i$ has a description as in the statement.
\end{proof}
\newpage
\section{Completely decomposed $\Theta_{\eta}$-divisors}\label{ChapcdTheta}
\subsection{Definition and geometric interpretation}
\begin{definition} Let $\eta$ be a polarization on $E^g$ with $\ker(\eta)$ connected. A \emph{completely decomposed $\Theta_{\eta}$-divisor} is a divisor of the form $D=\sum_{i=1}^g p^{n_i} A_i$ inducing the polarization $\eta$ where the $A_i$ are Abelian subvarieties of $E^g$ and $n_i\in \mathbb{N}$.
\end{definition}
\begin{definition}\label{DefGrschCompDec} Let $\eta$ be a polarization on $E^g$ with connected kernel. For a completely decomposed $\Theta_{\eta}$-divisor $D=\sum_{i=1}^g p^{n_i} A_i$ we consider the quotient maps $\xi_i:E^g\rightarrow E^g/A_i=E_i$. We define the \emph{group scheme associated to $D$} to be the scheme theoretic intersection $\cap_{i=1}^g \xi_i^{-1} \left(E_i[F^{n_i}]\right)$.
\end{definition}
The following lemma gives a geometric interpretation of a completely decomposed $\Theta_{\eta}$-divisor in terms of products of elliptic curves. This will explain their name.
\begin{lem}\label{AuxDivSuper} Let $\eta$ be a polarization on $E^g$ with $\ker(\eta)$ connected. A completely decomposed $\Theta_{\eta}$-divisor $D=\sum_{i=1}^g p^{n_i} A_i$ is invariant under the group scheme $\mathcal{H}$ associated to $D$. 
\par Furthermore, $\mathcal{H} \subset \ker(\eta)$ is a maximal isotropic subgroup scheme. The quotient $E^g/\mathcal{H}$ is isomorphic as a principally polarized abelian variety to a product of elliptic curves with the product polarization.
\end{lem}
\begin{proof}
$D$ is $\mathcal{H}$-invariant since for any $i\in \{1,\ldots, g\}$ the two Cartier divisors $p^{n_i} A_i$ and $\xi^{-1}\left(E_i[F^{n_i}]\right)$ have the same ideal and invariance of divisors is compatible with taking subgroups and adding Cartier divisors (the Cartier property is important here). The assertion that $\mathcal{H} \subset \ker(\eta)$ is isotropic follows from \cite[Proposition 3.10]{AP}.
\par We show now the second claim. Indeed, denote the quotient $E^g/\mathcal{H}$ by $A$ and the quotient map $E^g\rightarrow A$ by $\pi$. Since $\mathcal{H}\subset \ker(\eta)$ is isotropic, $\eta$ descends to a polarization $\lambda$. Notice that we do not yet know the degree of $\lambda$. By \cite[Proposition 3.10]{AP} $\lambda$ is induced by the divisor $\pi(D)$. Now the divisor $\pi(D)$ has the shape $\sum_{i=1}^g p^{n_i'} A_i'$ where $A_i'=\pi(A_i)$ is an abelian subvariety of $A$. Therefore we can define a group scheme $\mathcal{H}'=\cap (A/A_i')[F^{n_i'}]$. Then $\pi(D)$ is an $\mathcal{H}'$-invariant divisor. Thus $\pi^{-1}(\pi(D))$ is an $\pi^{-1}(\mathcal{H}')$-invariant divisor. But $D\subseteq \pi^{-1}(\pi(D))$ and by \cite[Proposition 3.10]{AP} these two divisors are rationally equivalent. This implies $D=\pi^{-1}(\pi(D))$.  But $D$ is a sum of abelian subvarieties and thus the only groups leaving $D$ invariant are the subgroups of the intersection of the components of $D$. Therefore $\pi^{-1}(\mathcal{H}')=\mathcal{H}$ and hence $\mathcal{H}'=\{0\}$. From that one infers that we must have $n_i'=0$ and $\cap A_i'=\{0\}$ scheme-theoretically. We claim that for any subset $S\in \{1, \ldots, g\}$ the intersection $\cap_{i\in S} A_i'$ is an abelian variety of codimension $\sharp S$. Indeed, $\cap_{i\in S} A_i'$ is connected since for any prime $l\neq p$ the $\mathbb{F}_l$ vector spaces $A_i'[l](\overline{k})\subset A[l](\overline{k})$ have codimension $2$ inside a $2g$-dimensional vector space and their total intersection is trivial. Thus they must be in general position. Similarly $\cap_{i\in S} A_i'$ is reduced since the Lie algebras $\text{Lie}(A_i') \subset \text{Lie}(A)$ are of codimension $1$ inside a $g$ dimensional $k$-vector space.
\par The claim implies that the $E_i'=\cap_{j\neq i} A_j'$ are elliptic curves. Furthermore, the claim implies that the map $\prod E_i'\rightarrow A$ induced from the inclusion is an isomorphism. This map sends the divisor $\pr^{-1}_i(0)$ to $A_i'$ and therefore identifies $\lambda$ with the product polarization on $\prod E_i'$. In particular $\lambda$ is a principal polarization and thus $\mathcal{H}\subset \ker(\eta)$ is a maximal isotropic subgroup.
\end{proof}
\begin{cor} $\mathcal{H}$ lifts to a level group in $\mathcal{G}(\mathcal{O}(D))$. 
\end{cor}
\begin{proof}
The $\mathcal{H}$-invariant divisor $D$ determines a level group in $\mathcal{G}(\mathcal{O}(D))$ by \cite[Proposition 3.10]{AP}.
\end{proof}
\begin{definition} Let $(D_1, D_2, \ldots, D_r)$ be a tuple of completely decomposed $\Theta_{\eta}$-divisors. For any $i\in \{1,\ldots, g\}$ denote by $\mathcal{H}_i$ the group scheme associated to $D_i$. We say that $(D_1, \ldots, D_r)$ is \emph{a spanning tuple of completely decomposed $\Theta_{\eta}$-divisors} if the map of commutative group schemes
$$\mathcal{H}_1 \times \ldots \times \mathcal{H}_r\rightarrow \ker(\eta)$$
induced from the inclusions $\mathcal{H}_i \subset \ker(\eta)$ is surjective and splits.
\end{definition}
\begin{exmpl}
For $g=2$ and $\ker(\eta)=E^2[F]$ completely decomposed $\Theta_{\eta}$-divisors are translates of reducible fibers of Moret-Bailly's family by \cite[Proposition 4.5, Corollary 4.6, Corollary 4.7]{AP}. This motivated their definition in the general case. Furthermore, Lemma \ref{AuxDivSuper} shows that completely decomposed $\Theta_{\eta}$-divisors have the same geometric interpretation in terms of products of elliptic curves generalizing the Moret-Bailly case to the general one.
\par In the Moret-Bailly case it is easy to see that a pair $(D_1,D_2)$ of completely decomposed $\Theta_{\eta}$-divisors is spanning if and only if $D_1\neq D_2$. In particular there exists always a spanning pair of completely decomposed $\Theta_{\eta}$-divisors in this case as a Moret-Bailly family has $5p-5\geqslant 2$ reducible fibers.
\end{exmpl}
\begin{rem}\label{ClassAuxDiv}
There is a classification of completely decomposed $\Theta_{\eta}$-divisors in terms of linear algebra over $\mathsf{O}$. This generalizes the results of \cite[Section 4.3]{AP}. In particular we have an algorithm computing all the completely decomposed $\Theta_{\eta}$-divisors for a given $\eta$. (The classification implies that this is a finite set.) We can also check whether a tuple of completely decomposed $\Theta_{\eta}$-divisors is spanning. See \cite[Section 8.4]{APThesis} for the details.
\end{rem}
\subsection{Rational equivalence of completely decomposed $\Theta_{\eta}$-divisors}\label{SecRatEquiv}
\begin{lem}\label{LemRatEq} Suppose $D_1, D_2$ are two completely decomposed $\Theta_{\eta}$-divisors. Then there is a 2-torsion point $P\in E^g[2]$ such that $t_P(D_1)\sim D_2$. There is an explicit formula for $P$ in terms of linear algebra over $\mathbb{F}_2$.
\end{lem}
\begin{proof}
Straightforward generalization of \cite[Lemma 4.15]{AP}. The proof works the same way.
\end{proof}
Furthermore, a rational function $\rho$ on $E^g$ such that $(\rho)=t_P(D_1)-D_2$ can be computed quickly by noticing that the whole situation can be described in terms of the geometry of elliptic curves. (See \cite[Section 8.6]{APThesis} for the details). It is worthwhile to mention that this method does not use Gr\"obner bases as the generic algorithm for computing such rational functions would do.
\section{Main results}\label{ChapMain}
\subsection{Notation}
We fix the following notation for the rest of the paper.
\begin{itemize}
\item $E/k$ is a supersingular elliptic curve defined over $\mathbb{F}_p$ such that the relative Frobenius satisfies $F^2+p=0$.
\item $\eta$ is a polarization on $E^g$ with connected kernel. Assume that there exists a spanning tuple of completely decomposed $\Theta_{\eta}$-divisors $D_1, \ldots, D_r$.
\item $\mathsf{O}=\text{End}(E)$.
\item There is an additive bijection
$$\gamma: \text{Mat}_{m,n}(\mathsf{O})\stackrel{\sim}{\rightarrow}\text{Hom}(E^n,E^m) $$
$$ \Phi=(\varphi_{ij})\mapsto \left(\psi:E^n\rightarrow E^m,\, (P_1,\ldots, P_n)\mapsto \left(\sum_{j=1}^n \varphi_{1j}(P_j)), \ldots, \sum_{j=1}^n \varphi_{mj}(P_j)\right)\right)$$
This bijection turns matrix multiplication into composition of maps. Notice that we use the convention that maps should be composed from right to left.
\item $\mu: E^g\rightarrow (E^g)^\vee$ denotes the natural product polarization.
\item There is a bijective map from the set of positive definite hermitian matrices in $\Mat_{g,g}(\mathsf{O})$ to the set of polarizations on $E^g$ given by
$$H\mapsto (\mu\circ \gamma(H):E^g\rightarrow (E^g)^\vee)$$
Denote by $H$ the preimage of $\eta$ under this bijection.
\item The letters $D, \mathcal{H},\tilde{\mathcal{H}}, A, \mathcal{L} $ will also have a fixed meaning which is explained below.
\end{itemize}
\subsection{Overview of the algorithm}
In this section we outline the main algorithm. The details of the individual steps will be given in the next sections. Beware that we anticipate some of the notation and results from the subsequent sections. Thus the reader might want to repeatedly come back here at a later stage.
\begin{itemize}[leftmargin=1.5cm]\label{AlgMain}
\item[\textbf{Input:}]
\begin{itemize}
\item[i)] The elliptic curve $E$. Assume that all the points of $E[8]$ are defined over $k$.
\item[ii)] An hermitian matrix $H\in M_{g,g}$ describing the polarization $\eta$. Assume that there exists a spanning tuple of completely decomposed $\Theta_{\eta}$-divisors $D_1, \ldots, D_r$.
\item[iii)] The spanning tuple will be given as an input by specifying for each $D_j$ a tuple of matrices $H_i\in M_{g,g}(\mathsf{O}),\, i=1,\ldots, g$ (see \cite[Proposition 8.6]{APThesis}).
\item[iv)]  A maximally isotropic Dieudonn\'e submodule $M(\mathcal{H})\subset M(\ker(\eta))$.
\end{itemize}
\item[\textbf{Output:}] The level $4$ theta nullvalues $\theta_{a,b}$ of $A=E^g/\mathcal{H}$.
\item[\textbf{Method:}]
\item[(1)] Compute a basis $x_1, \ldots, x_g, y_1, \ldots, y_g$ of $E^g[4](k)$ symplectic with respect to the pairing of Lemma \ref{LemTwoTorsTheta}.
\item[(2)] Compute a $2$-torsion point $P\in E^g[2](k)$ such that $e_*^{t_P^{-1}(D_r)}$ is in normal form with respect to the basis $2x_1, \ldots, 2x_g, 2y_1, \ldots, 2y_g$ of $E^g[2](k)$.
\item[(3)] Compute the rational functions $\rho_{x_i}$ from Theorem \ref{ThmTwoTorsLevel} (and the analogously defined rational functions for the other level group $\tilde{V}_2$).
\item[(4)] Compute with Lemma \ref{LemRatEq} the $2$-torsion points $P_1, \ldots, P_r \in E^g[2](k)$ such that
$$t_{P_i}^{-1}(D_i) \sim t_P^{-1}( D_r)\stackrel{\text{def.}}{=} D\,.$$
\item[(5)] Compute rational functions $\rho_i, i=1, \ldots, r-1$ such that
$$(\rho_i)= t_{P_i}^{-1}(D_i)-D$$
\item[(6)] Use Theorem \ref{ThmInvConn} to calculate $\rho_{\tilde{\mathcal{H}}}$, a generator of the one-dimensional vector space
$$\HH^0(E^g, \mathcal{O}(D))^{\tilde{\mathcal{H}}}\,.$$
(This will use the the rational functions $\rho_i$ to compute the action of the connected level group.)
\item[(7)] Compute $\theta_{a,b}$ from $\rho_{\tilde{\mathcal{H}}}$ via Theorem \ref{ThmThetaNull}. (In this step we use the rational functions $\rho_{x_i}$ to compute the action of the \'etale level groups.)
\item[\textbf{Return:}] $\theta_{a,b}$.
\end{itemize}
\phantom{}
\subsection{Construction of \'etale level groups}\label{SecEtaleLevGr}
Let $N$ be any positive even integer with $p \nmid N$. Later we will apply the results of this section with $N$ equal to $4$. Assume that all the $2N$-torsion points of $E$ are defined over $k$. The existence of the Weil pairing implies that $k$ contains a primitive $N$-th root of unity. Choose and fix such a root of unity $\zeta_N\in k$. This determines an isomorphism $\mu_N\cong \underline{\mathbb{Z}/N\mathbb{Z}}$.
\par We will give an explicit description of two level subgroups in $\mathcal{G}(\mathcal{O}(ND))$ of exponent $N$. Here $D$ will be a suitable translate of $D_r$. These level subgroups will later be used to construct a $\Gamma(4,8)$-level structure on the quotient.
\begin{lem}\label{LemTwoTorsTheta} $K(\mathcal{O}(N D_r))[N]=E^g[N]$. The restriction of the commutator pairing is given by $[x,y]=(x, \gamma(H)(y))$, where $(\cdot, \cdot)$ is the pairing on $E^g[N]$ induced from the Weil pairing on $E[N]$.
\end{lem}
\begin{proof}
$K(\mathcal{O}(N D_r))[N]=E^g[N]$ follows from Lemma \ref{KMultn}. To obtain the commutator pairing use \cite[p. 211, formulas (4), (5)]{MumAV}.
\end{proof}
We are now ready to explain the construction of two level groups:
\par Choose two disjoint subspaces $V_1,V_2\subset E^g[N](k)$ maximally isotropic with respect to the pairing from \ref{LemTwoTorsTheta}. Choose a basis $x_1, \ldots, x_g$ of $V_1$. Denote the dual basis of $V_2$ by $y_1, \ldots, y_g$. Choose points $x_i',\, y_i'\in E^g[2N](k)$ with $x_i=2 x_i',\, y_i=2 y_i'$.
\par Since $N$ is even, the $x_1,\ldots, x_g, y_1,\ldots, y_g$ also determine a basis of $E^g[2](k)$. Because $2\nmid \deg(D)$, the same argument as in Lemma \ref{Gamma48Can} shows the existence of a 2-torsion point $P\in E^g[2](k)$ such that $e_*^{t_P^{-1}(D_r)}$ is given by the matrix
$$\begin{pmatrix}
0 & I_g\\
0 & 0
\end{pmatrix}$$
with respect to our basis of $E^g[2](k)$. Put $D=t_P^{-1}(D_r)$.
\par On the other hand the divisors $2D$ and $t_{-x_i'-jx_i}^*(D)+t_{x_i'+jx_i}^*(D)$ (resp. $t_{-y_i'-j y_i}^*(D)+t_{y_i'-j y_i}^*(D)$) are rationally equivalent for $i=1,\ldots, g$ and $j=1, \ldots, \frac{N}{2}$ by the theorem of the square. Choose rational functions $\rho_{x_i,j}$ such that $(\rho_{x_i,j})= t_{-x_i'-jx_i}^*(D)+t_{x_i'+jx_i}^*(D) -2D$, resp. $\rho_{y_i,j}$ such that $(\rho_{y_i,j})= t_{-y_i'-jy_i}^*(D)+t_{y_i'+jy_i}^*(D) -2D$.
\begin{thm}\label{ThmTwoTorsLevel} 
In the notation preceding the proposition we have:
\begin{itemize}
\item[i)] If we define
$$\rho_{x_i}=\prod_{j=1}^{\frac{N}{2}}\rho_{x_i,j}^{-1} t_{x_i}^*(\rho_{x_i,j})$$
then multiplication by $\rho_{x_i}$ gives an isomorphism $\mathcal{O}(N D)\rightarrow t_{x_i}^*\mathcal{O}(ND)$ denoted $\phi_{x_i}$.
\item[ii)] The elements $(x_i, \phi_{x_i})\in \mathcal{G}(\mathcal{O}(ND))(k)$ generate a level subgroup denoted $\tilde{V_1}$.
\item[iii)] For all $g\in \tilde{V_1}$ one has $\delta_{-1}(g)=g^{-1}$.
\item[iv)] Similarly we get a level subgroup $\tilde{V_2}$ with the analogous properties.
\end{itemize}
\end{thm}
\begin{proof}
For ii) notice denote by $K_i$ the subgroup of $E^g[N]$ generated by $x_i$. The divisor
$$\sum_{j=1}^{\frac{N}{2}} (t_{-x_i'-jx_i}^*(D)+t_{x_i'+jx_i}^*(D))=\sum_{j=1}^{N} t_{x_i'+jx_i}^*(D)$$
 is $K_i$ invariant. By \cite[Proposition 3.10]{AP} we obtain a level subgroup above $K_i$. Furthermore, these level groups will commute with each other as $V_1$ is isotropic. Therefore these level groups will generate a level group denoted $\tilde{V_1}$. This proves ii)
\par It is easy to see that the elements of the level group $\tilde{V_1}$ have a description as in i). This proves i).
\par One can check with a small computation that the rational functions $\rho_{x_i,j}$ are symmetric, i.e. satisfy $[-1]^*\rho_{x_i,j}=\rho_{x_i,j}$. This implies iii) by unraveling the definition of $\delta_{-1}$.
\par Part iv) is clear.
\end{proof}
\begin{cor} The action of $\tilde{V_1}$ on $\HH^0(E^g, \mathcal{O}(ND)$ is given by
$$U_{(x_i, \phi_{x_i})} (\rho)=\rho_{x_i} t_{x_i}^*(\rho_{x_i}^{-1} \rho)$$
\end{cor}
\begin{proof}
Clear.
\end{proof}
Therefore we can fully understand the action of the level groups $\tilde{V_1},\tilde{V_2}$ provided we can compute the rational functions $\rho_{x_i,j},\, \rho_{y_i,j}$. These rational functions can be computed alongside the rational functions of subsection \ref{SecRatEquiv}. See \cite[Section 8.4]{APThesis} for more details.
\subsection{Setting up the Dieudonn\'e modules}\label{SecSettDieu}
Choose an $n$ such that $\ker(\eta)\subseteq E^g[F^n]$. Then all the group schemes we need to consider are killed by $F^n, V^n$. Therefore we can work with the Dieudonn\'e modules $M_{n,n}(\cdot)$. For shorthand notation we will omit $n,n$ in the index.
\begin{lem}\label{WittSurj}
There exists a surjection of group schemes over $\mathbb{F}_p$
$$W_{n,n}\twoheadrightarrow E[F^n]$$
\end{lem}
\begin{proof}
Denote by $M=M(E[F^n])$. This is a $D_{n,n}$-module. But $D_{n,n}$ is a local ring with maximal ideal $(F,V)$. On the other hand
$$\frac{M}{(F,V)M}\cong \Hom(E[F^n],  \alpha_p)\cong\Hom(\alpha_p, E[V^n])\cong k$$
where in the second equation we used Cartier duality. Thus Nakayama's lemma implies that $M$ is generated as a $D_{n,n}$-module by one element. Therefore there is a surjection $D_{n,n} \twoheadrightarrow M$. Since $M(\cdot)$ is fully faithful, the lemma follows.
\end{proof}
\begin{rem}
There is an algorithm that computes a surjection as in the previous lemma. The idea of the construction is to lift the elliptic curve $E$ to an elliptic curve $\mathcal{E}/\mathbb{Z}_p$. The formal group $\hat{\mathcal{E}}$ is a Lubin-Tate formal group. This implies that $\hat{\mathcal{E}}$ is isomorphic to a formal group  $\widehat{ \text{LT}}$ that can be explicitly described using the Witt vector construction. Using the logarithm we can compute an isomorphism $\widehat{\text{LT}}\cong \hat{\mathcal{E}} $ with suitable precision. Reducing mod $p$ gives a surjection $W_{n,n}\twoheadrightarrow E[F^n]$. For the details see \cite[Section 6.3]{APThesis}.
\end{rem}
Let us apply this remark and compute a surjection $W_{n,n}\rightarrow E[F^n]$ for some $n$ such that $\ker(\eta)\subseteq E^g[F^n]$. The surjection $W_{n,n}\rightarrow E[F^n]$ determines an element $\delta\in M(E[F^n])$. The Dieudonn\'e module $M(E[F^n])$ is generated as a $W(k)$ module by $\delta, F\delta$.
\par On the other hand the endomorphism algebra $\mathsf{O}$ acts on $M(E[F^n])$. The action on the generators $\delta, F\delta$ can be made explicit (see \cite[Proposition 6.22]{APThesis} for the details). We can calculate a $\mathbb{Z}$-linear map
$$\Psi: \mathsf{O}\rightarrow \Mat_{2,2}(W(k))$$
such that for all $\alpha\in \mathsf{O}$ one has $M(\alpha)(\delta)=\Psi(\alpha)_{11}\delta+\Psi(\alpha)_{21}F\delta,\, M(\alpha)(F\delta)=\Psi(\alpha)_{12}\delta+\Psi(\alpha)_{22}F\delta$.
\par This map extends in the obvious way to a map on matrices $\Psi: \Mat_{s,t}(\mathsf{O})\rightarrow \Mat_{2s,2t}(W(k))$ by building a block matrix consisting of $2\times 2$ blocks. The block matrices of this shape will be used to perform all kinds of linear algebra with the Dieudonn\'e modules.
\par The next step is to calculate the submodule $M(\ker(\eta)) \subseteq M(E[F^n])$. This is given by the following lemma:
\begin{lem}
Denote by $M(H): M(E[F^n])^g \rightarrow M(E[F^n])^g$ the map induced by the matrix $H$. Then $M(\ker(\eta))=\ker(M(H))$.
\end{lem}
\begin{proof}
By definition of $H$ one has $\eta=\mu\circ \gamma(H)$ where $\mu: E^g\rightarrow (E^g)^\vee$ is the natural product polarization. Since $\mu$ is bijective, we have $\ker(\eta)=\ker (\gamma(H))$. Now $M(H)$ is the map on Dieudonn\'e modules induced by $\gamma(H)_{|E^g[F^n]}: E^g[F^n]\rightarrow E^g[F^n]$. Thus by the exactness of $M(\cdot)$ one has $\ker(M(H))=M(E^g[F^n]\cap \ker(\eta))$. But since we assumed $\ker(\eta)\subseteq E^g[F^n]$, we conclude  $M(\ker(\eta))=\ker(M(H))$.
\end{proof}
The algorithm needs as an input the Dieudonn\'e module of a maximal isotropic subgroup scheme $\mathcal{H}\subset \ker(\eta)$. To make sense of this we need to identify the pairing on $M(\ker(\eta))$ induced by the commutator pairing $\ker(\eta)\times \ker(\eta)\rightarrow \mathbb{G}_m$
\begin{lem}\label{DieuPair}
There is a constant $c\in W(k)^\times$ such that the $W(k)$ linear alternating pairing 
$$M(\ker(\eta))\times M(\ker(\eta))\rightarrow W(k)\left[\frac{1}{p}\right]/W(k)$$
is given by the matrix
$$c p^{-n} J \Psi(H)$$
where
$$J=\begin{pmatrix}
0 & 1\\
-1 & 0\\
&& \ddots\\
&&& 0 & 1\\
&&&-1 & 0\\
\end{pmatrix}$$
\end{lem}
\begin{proof}
There is an inclusion $M(E[F^n])\subseteq M(E[p^{\lceil \frac{n}{2} \rceil}])$. Now $M(E[p^{\lceil \frac{n}{2} \rceil}])$ carries a Weil pairing denoted $(\cdot, \cdot): M(E[p^{\lceil \frac{n}{2} \rceil}])\times M(E[p^{\lceil \frac{n}{2} \rceil}])\rightarrow W(k)[\frac{1}{p}]/W(k)$. We claim that $(\delta, F\delta)=c p^{-\lfloor \frac{n}{2} \rfloor}$ for some $c\in W(k)^\times$. Indeed, if $n$ is even, then $M(E[F^n])=M(E[p^{\lceil \frac{n}{2} \rceil}])$ and the claim follows from the non-degeneracy of the Weil pairing. If $n$ is odd, then there is a $\delta_0\in M(E[p^{\lceil \frac{n}{2} \rceil}])$ such that $p\delta_0=F\delta$. Furthermore, $\delta_0, \delta$ generate $M(E[p^{\lceil \frac{n}{2} \rceil}])$ as a $W(k)$ module. Again the claim follows from the non-degeneracy of the Weil pairing and an easy computation.
\par To compare the Weil pairing with the commutator pairing we can use \cite[p. 211, formulas (4), (5)]{MumAV} as in the proof of \ref{LemTwoTorsTheta}. (Mumford assumes that $p$ should not divide the degree of the polarization but the formulas are still valid in this case by \cite[Proof of Corollary 11.22]{vdGMooAV})
\end{proof}
\subsection{Invariants of the connected level group}\label{SecInvCon}
Recall that the \hyperref[AlgMain]{main algorithm} takes as an input a maximally isotropic Dieudonn\'e submodule $M(\mathcal{H})\subset M(\ker(\eta))$. With Lemma \ref{DieuPair} we can check whether $M(\mathcal{H})$ really is isotropic.
\par On the other hand, since $M(\cdot)$ is an equivalence of categories, there is a corresponding maximally isotropic subgroup scheme $\mathcal{H}\subset \ker(\eta)$. We will use the following notation for the rest of the paper: Define $A$ to be $E^g/\mathcal{H}$ and denote by $\pi: E^g\rightarrow A$ the quotient map. Notice that $\mathcal{H}$ is unipotent and isotropic in $\ker(\eta)$. Recall further that $D$ induces $\eta$. Therefore there is a unique level group $\tilde{\mathcal{H}}\subset \mathcal{G}(\mathcal{O}(D))$ above $\mathcal{H}$ (well known; see also \cite[Corollary 3.5]{AP}). This implies that the line bundle $\mathcal{O}(D)$ descends to a line bundle, say $\mathcal{L}$ on $A$. One has $\deg(\mathcal{L})=1$ since $\mathcal{H}$ is maximally isotropic. Our goal is to compute the theta nullvalues of the principally polarized abelian variety $(A, \mathcal{L})$.
\par Since all our considerations of theta nullvalues start with the 1-dimensional $\HH^0(A, \mathcal{L})$, we need to compute the image of
$$\pi^*: \HH^0(A, \mathcal{L})\rightarrow \HH^0(E^g, \mathcal{O}(D))$$
But this image is given by $\HH^0(E^g, \mathcal{O}(D))^{\tilde{\mathcal{H}}}$ (well known; see also \cite[Proposition 3.18]{AP}).  In the Moret-Bailly case that one dimensional vector space is generated by the rational function $g_x$ of \cite[Theorem 4.19]{AP} (see the proof of the quoted theorem). The reader who is only interested in the computation of Moret-Bailly families can skip the rest of this section entirely.
\par First we want to identify $\tilde{\mathcal{H}}$. Recall that we assumed the spanning tuple $D_1,\ldots, D_r$ of completely decomposed $\Theta_{\eta}$-divisors inducing $\eta$. The completely decomposed $\Theta_{\eta}$-divisors $D_1,\ldots, D_r$ determine group schemes $\mathcal{H}_i\subset \ker(\eta)$ and level groups $\tilde{\mathcal{H}}_i \subset \mathcal{G}(\mathcal{O}(D))$.
\begin{definition} We define the \emph{shuffle pairing} on $\tilde{\mathcal{H}}_1\times\ldots\times\tilde{\mathcal{H}}_r$ to be the pairing
$$\sh:(\tilde{\mathcal{H}}_1\times\ldots\times\tilde{\mathcal{H}}_r)\times (\tilde{\mathcal{H}}_1\times\ldots\times\tilde{\mathcal{H}}_r)\rightarrow \mathbb{G}_m  $$
$$(h_1,\ldots, h_r),(h_1',\ldots, h_r') \mapsto h_1\cdots h_r\cdot h_1'\cdots h_r' \cdot ((h_1\cdot h_1')\cdots (h_r \cdot h_r') )^{-1}$$
\end{definition}
\begin{lem}\label{Biad} The shuffle pairing is biadditive.
\end{lem}
\begin{proof}
Indeed, the shuffle pairing can be written as the product of restrictions of the commutator pairing.
\end{proof}
The following Theorem gives $\tilde{\mathcal{H}}$:
\begin{thm}\label{ShuffleLevel} Let $\mathcal{H}\subset \ker(\eta)$ be an isotropic subgroup scheme. For any additive splitting $\sigma: \ker(\eta)\rightarrow \mathcal{H}_1\times \ldots \times\mathcal{H}_r$ the image of
$$\mathcal{H} \rightarrow \mathcal{G}(\mathcal{O}(D)),\, h\mapsto \sh\left(\sigma(h), \frac{\sigma(h)}{2}\right) \tau(h)$$
is $\tilde{\mathcal{H}}$. Here $\tau$ denotes the composition
$$\ker(\eta)\stackrel{\sigma}{\rightarrow} \mathcal{H}_1\times \ldots \times\mathcal{H}_r\cong \tilde{\mathcal{H}_1}\times \ldots \times\tilde{\mathcal{H}_r} \stackrel{(h_1,\ldots, h_r) \mapsto h_1\cdots h_r}\longrightarrow \mathcal{G}(\mathcal{O}(D))$$
\end{thm}
\begin{proof}
Indeed, clearly it suffices to show that the map
$$j:\mathcal{H} \rightarrow \mathcal{G}(\mathcal{O}(D)),\, h\mapsto \sh\left(\sigma(h), \frac{\sigma(h)}{2}\right) \tau(h)$$
is a group homomorphism. To prove this, consider any $k$-scheme $T$ and $h, h'\in \mathcal{H}(T)$ arbitrary. We first claim that $\sh\left(\sigma(h),\sigma(h')\right)=\sh\left(\sigma(h'),\sigma(h)\right)$. Indeed, denote the components of $\sigma(h)$ by $(h_1,\ldots, h_r)$, resp. $\sigma(h')=(h_1', \ldots, h_r')$. Now $\sigma$ splits the map $\mathcal{H}_1\times\ldots \times \mathcal{H}_r\rightarrow \ker(\eta), \, (h_1, \ldots, h_r)\mapsto \sum_{i=1}^r h_i$. Therefore the element $h_1 \cdots h_r \in \mathcal{G}(\mathcal{O}(D))$ maps to $h\in \mathcal{H}\subset K(\mathcal{O}(D))$, resp. $h_1'  \cdots h_r' \in \mathcal{G}(\mathcal{O}(D))$ maps to $h'$ under the map $\mathcal{G}(\mathcal{O}(D))\rightarrow K(\mathcal{O}(D))$. Since $\mathcal{H}$ is by assumption isotropic, we have $[h,h']=1$. Thus the elements $h_1 \cdots h_r \in \mathcal{G}(\mathcal{O}(D))$ and $h_1'  \cdots h_r' \in \mathcal{G}(\mathcal{O}(D))$ commute. This implies $\sh\left(\sigma(h),\sigma(h')\right)=\sh\left(\sigma(h'),\sigma(h)\right)$ immediately from the definition of the shuffle pairing.
\par We are now ready to show that $j$ is a group homomorphism. Indeed,
$$j(h)j(h')= \sh\left(\sigma(h), \frac{\sigma(h)}{2}\right) \tau(h) \sh\left(\sigma(h'), \frac{\sigma(h')}{2}\right) \tau(h')=$$
$$\sh\left(\sigma(h), \frac{\sigma(h)}{2}\right)  \sh\left(\sigma(h'), \frac{\sigma(h')}{2}\right) \tau(h)\tau(h')=$$
$$\sh\left(\sigma(h), \frac{\sigma(h)}{2}\right)  \sh\left(\sigma(h'), \frac{\sigma(h')}{2}\right) \sh\left(\sigma(h), \sigma(h')\right)\tau(h+h')=$$
$$\sh\left(\sigma(h), \frac{\sigma(h)}{2}\right)  \sh\left(\sigma(h'), \frac{\sigma(h')}{2}\right) \sh\left(\sigma(h), \frac{\sigma(h')}{2}\right)\sh\left(\sigma(h'), \frac{\sigma(h)}{2}\right)\tau(h+h')=$$
$$\sh\left(\sigma(h+h'), \frac{\sigma(h+h')}{2}\right) \tau(h+h')=j(h+h') $$
Here it was important that $\sigma$ is a group homomorphism and $\sh$ is biadditive by Lemma \ref{Biad}.
\end{proof}
We need to introduce some further notation:
\begin{itemize}
\item By Lemma \ref{LemRatEq} there exist $P_1, \ldots, P_r\in E^g[2](k)$ such that $t_{P_i}^{-1}(D_i)$ is rationally equivalent to $D$. Define $D_i'$ as $t_{P_i}^{-1}(D_i)$. Then $D_i'\sim D$ for all $i$.
\end{itemize}
We will now use Dieudonne theory and Theorem \ref{ShuffleLevel} to make the action of $\tilde{\mathcal{H}}$ on $\HH^0(E^g, \mathcal{O}(D))$ explicit. Indeed, we are given the Dieudonn\'e module $M(\mathcal{H})$. Since $\mathcal{H}$ is killed by $F^n, V^n$, there exists a surjection of Dieudonn\'e modules $M(W_{n,n})^a\rightarrow M(\mathcal{H})$ for some $a\in \mathbb{N}$. In fact we can assume that $M(\mathcal{H})$ is given as an input by specifying a set of elements in $M(E[F^n])^g$ generating $M(\mathcal{H})$ as a Dieudonn\'e module. This determines such a surjection.
\par On the other hand there is a corresponding surjection of group schemes $W_{n,n}^a\rightarrow \mathcal{H}$. Denote by $h=\mathcal{H}(W_{n,n}^a)$ the corresponding scheme valued point.
\par We want to make the composition $W_{n,n}^a\rightarrow \mathcal{H}\cong \tilde{\mathcal{H}}\subset \mathcal{G}(\mathcal{O}(D))$ explicit. This amounts to plugging $h$ into the formula of Theorem \ref{ShuffleLevel}. But, instead of working with the theta group directly we will only describe the induced action of $W_{n,n}^a$ on $\HH^0(E^g, \mathcal{O}(D))$. Thus we want to know the comodule map $c_W: \HH^0(E^g, \mathcal{O}(D))\rightarrow \HH^0(E^g, \mathcal{O}(D))\otimes A_{W_{n,n}^a}$. Let $\rho\in \HH^0(E^g, \mathcal{O}(D))$ be arbitrary. By Theorem \ref{ShuffleLevel} $c_W(\rho)=\sh(\sigma(h), \frac{\sigma(h)}{2}) \tau(h) \cdot \rho$. We first figure out what $\tau(h)\cdot \rho$ is. Indeed, by assumption the map $\mathcal{H}_1\times \ldots \times \mathcal{H}_r\rightarrow \ker(\eta)$ is split surjective. On Dieudonn\'e modules this map induces the map $M(\mathcal{H}_1)\oplus\ldots \oplus M(\mathcal{H}_r)\rightarrow M(\ker(\eta)),\, (m_1,\ldots, m_r)\mapsto \sum_{i=1}^r m_i$. This map must also be split surjective as $M(\cdot)$ is fully faithful. Compute a splitting denoted $M(\sigma)$ (you need to know the Dieudonn\'e module $M(\mathcal{H}_i)\subset M(\ker(\eta))$; it can be obtained using the classification of c.d. $\Theta_\eta$-divisors). Now $M(\sigma)$ must be induced by a map $\sigma: \ker(\eta)\rightarrow \mathcal{H}_1\times\ldots \times \mathcal{H}_r$.
Consider the following composition:
 $$W_{n,n}^a \rightarrow \mathcal{H} \rightarrow \ker(\eta)\stackrel{\sigma}{\rightarrow}\mathcal{H}_1\times \ldots \times \mathcal{H}_r \stackrel{\pr_i}{\rightarrow} \mathcal{H}_i $$ 
We claim that this map can be made explicit. Indeed, we can compute the composition of maps of Dieudonn\'e modules $M(W_{n,n})^a\rightarrow M(\mathcal{H}_i)\subset M(E^g[F^n])$. Consider the diagram
$$\begin{tikzcd}
&  M(W_{n,n})^g \arrow[d, two heads]\\
M(W_{n,n})^a \arrow{r} & M(E^g[F^n])
\end{tikzcd}$$
Since $M(W_{n,n})^a$ is a free $\frac{D(k)}{D(k) F^n + D(k) V^n}$-module, there exists a map of Dieudonn\'e modules $M(W_{n,n}^a)\rightarrow M(W_{n,n}^g)$ making the diagram above commutative. This map can be easily computed by lifting the images of basis elements of $M(W_{n,n})^a$.
\par On the other can explicitly write down maps $W_{n,n}\rightarrow W_{n,n}$ via the isomorphism $\End(W_{n,n})\cong \End(M(W_{n,n}))=\frac{D(k)}{D(k) F^n+D(k) V^n}$. Therefore we can compute polynomials describing the map of group schemes
$$W_{n,n}^a\rightarrow W_{n,n}^g\rightarrow E^g[F^n]$$
with image $\mathcal{H}_i$.
\begin{lem}
The scheme valued point in $\mathcal{H}_i(W_{n,n}^a)$ corresponding to the map $W_{n,n}^a\rightarrow \mathcal{H}_i$ is the $h_i$ from Theorem \ref{ShuffleLevel}.
 \end{lem}
 \begin{proof}
Clear by construction.
 \end{proof}
To figure out $U_\tau(h)(\rho)=U_{h_1\cdots h_r}(\rho)$ notice that by \cite[Proposition 3.18]{AP} the level group above $\mathcal{H}_i$ acts on $\HH^0(E^g, \mathcal{O}(D_i'))$ via translating rational functions. Since the rational equivalence $D \sim D_i'$ comes from a rational function that we can compute explicitly, we can compute $U_{\tau(h)}( \rho)\in \HH^0(E^g, \mathcal{O}(D))\otimes A_{W_{n,n}^a}$. The calculation starts from the translation map on rational functions $K(E^g)\rightarrow K(E^g)\otimes A_{E^g[F^n]}$ induced from the group law on $E$. The computation uses only the following operations: Evaluation of the translation map, composition of polynomial expressions, multiplication by rational functions.
\par To complete our evaluation of the formula of Theorem \ref{ShuffleLevel} we need the element $\sh\left( \sigma(h), \frac{\sigma(h)}{2} \right) \in \mathbb{G}_m(W_{n,n}^a)=A_{W_{n,n}^a}^\times$. This is provided by the following lemma:
\begin{lem}
For any $\rho\in \HH^0(E^g, \mathcal{O}(D))$. The two elements
$$U_{h_1 \cdots h_r}( \rho),\, U_{h_r \cdots h_1}( \rho) \in \HH^0(E^g, \mathcal{O}(D))\otimes A_{W_{n,n}^a}$$
satisfy
$$ U_{h_r \cdots h_1}(\rho)=\sh\left(\sigma(h), \sigma(h)\right) U_{h_1 \cdots h_r}( \rho)$$
\end{lem}
\begin{proof}
We have $U_{h_r \cdots h_1}( \rho)= U_{h_r\cdots h_1 (h_r^{-1} \cdots h_1^{-1})} \left( U_{h_1 \cdots h_r} \left( \rho \right)\right)$. But the element $h_r\cdots h_1 (h_r^{-1} \cdots h_1^{-1})\in \mathcal{G}(\mathcal{O}(D))(W_{n,n})$ lies in $\mathbb{G}_m$. An elementary computation with commutators shows:
$$h_r\cdots h_1 (h_r^{-1} \cdots h_1^{-1})=(h_1\cdots h_r) (h_1 \cdots h_r) (h_r^{-2} \cdots h_1^{-2})=\sh\left(\sigma(h), \sigma(h)\right)  \,.$$
Since the representation of $\mathcal{G}(\mathcal{O}(D))$ has weight $1$, the lemma follows.
\end{proof}
This lemma leads to the following method for calculating $\sh\left( \sigma(h), \frac{\sigma(h)}{2} \right)$: We calculate the two elements $U_{h_1 \cdots h_r}( \rho),\, U_{h_r \cdots h_1}(\rho) \in \HH^0(E^g, \mathcal{O}(D))\otimes A_{W_{n,n}^a}$ with the method described above. Their quotient gives us $\sh\left( \sigma(h), \sigma(h) \right)$. We obtain $\sh\left( \sigma(h), \frac{\sigma(h)}{2} \right)$ by taking the square root.
\par To summarize our discussion above: We can evaluate the comodule map $c_W: \HH^0(E^g, \mathcal{O}(D))\rightarrow \HH^0(E^g, \mathcal{O}(D))\otimes A_{W_{n,n}^a}$ explicitly. This gives a description of the action of $W_{n,n}^a$ via the composition
$$W_{n,n}^a\rightarrow \mathcal{H} \cong \tilde{\mathcal{H}}\subset \mathcal{G}(\mathcal{O}(D))$$
Using the comodule map we can calculate an invariant vector $\rho_{\tilde{\mathcal{H}}}\in \HH^0(E^g, \mathcal{O}(D))^{W_{n,n}^a}$ with the method of section \ref{Repgrsch}. Thus we conclude:
\begin{thm}\label{ThmInvConn}
In the notation above $\rho_{\tilde{\mathcal{H}}}$ is a generator of the one dimensional vector space $\HH^0(E^g, \mathcal{O}(D))^{\tilde{\mathcal{H}}}$.
\end{thm}
\begin{proof}
Since $W_{n,n}^a\rightarrow \mathcal{H}$ is surjective, we have $\HH^0(E^g, \mathcal{O}(D))^{\tilde{\mathcal{H}}}=\HH^0(E^g, \mathcal{O}(D))^{W_{n,n}^a}$.
\end{proof}
\subsection{Computation of Theta Nullvalues}
In this section we explain how to equip $(A, \mathcal{L})$ with a $\Gamma(4,8)$ level structure. Thereafter we will describe an algorithm that computes the resulting theta nullvalues $\vartheta_{a,b}$. Apply Theorem \ref{ThmTwoTorsLevel} with $N=4$ to obtain level groups $\tilde{K}_1\,, \tilde{K}_2\subset \mathcal{G}(\mathcal{O}(4D))$.
\begin{lem}
There are level groups $\tilde{K}_1\,, \tilde{K}_2\in \tG{4}$ such that:
\begin{itemize}
\item[i)] $\tilde{K}_1\,, \tilde{K}_2$ pull back to $\tilde{V}_1,\, \tilde{V}_2$ under $\pi$,
\item[ii)] $\tilde{K}_1\,, \tilde{K}_2$ determine a symmetric $\vartheta$-structure $f_2$ for $\tG{4}$,
\item[iii)] $(A, \lambda_{\mathcal{L}}, f_2)$ is a p.p.a.v. with a $\Gamma(4,8)$ level structure. Furthermore, $\mathcal{L}$ is the unique line bundle determined from Lemma \ref{Gamma48Can}.
\end{itemize}
\end{lem}
\begin{proof}
The line bundle $\mathcal{L}^4$ satisfies $\pi^*\mathcal{L}^4\cong \mathcal{O}(4D)$. This determines a level subgroup above $\mathcal{H}=\ker(\pi)$ which we denote by abuse of notation also by $\tilde{\mathcal{H}}\subset \mathcal{G}(\mathcal{O}(4D))$. Denote by $\mathcal{C}(\tilde{\mathcal{H}})\subset \mathcal{G}(\mathcal{O}(4D))$ the centralizer of $\mathcal{H}$. Since $\mathcal{H}$ is killed by a power of $p$ and $\tilde{V}_1, \tilde{V}_2$ have exponent $4$, one has $\tilde{V}_1, \tilde{V}_2\subset \mathcal{C}(\tilde{\mathcal{H}})$. But \cite[Proposition 8.15]{vdGMooAV} implies $\tG{4}\cong \mathcal{C}(\tilde{\mathcal{H}})/\tilde{\mathcal{H}}$. Define $\tilde{K}_1$ (resp. $\tilde{K}_2$) to be the image of $\tilde{V_1}$ (resp. $\title{V_2}$) under this isomorphism. This shows i).
\par To show ii) notice that $\tilde{K}_1 \cap \tilde{K}_2=\{1\}$. Furthermore, $\sharp \tilde{K}_1=\sharp \tilde{K}_2=4^{g}= \sqrt{ \sharp K(\mathcal{L}^4)}$. Therefore $\tilde{K}_1, \tilde{K}_2$ are two maximal level subgroups in $\tG{4}$. These determine a $\vartheta$-structure $f_2$ for $\tG{4}$. The symmetry of $f_2$ follows from Theorem \ref{ThmTwoTorsLevel} iv).
\par Let us prove iii). Indeed, by construction of $D$ and Proposition \ref{QFProp} iii) the quadratic form $e_*^\mathcal{L}$ is in normal form with respect to our basis of $A[2](k)$. This proves iii).
\end{proof}
We are now ready to explain the computation of $\vartheta_{a,b}$. Indeed, first calculate a rational function $\rho_{\mathfrak{2}}\in K(E^g)$ such that $(\rho_{\mathfrak{2}})=[2]^{-1}(D)-4D$. Next calculate a local equation for $D$ at $0$, i.e. a rational function $u\in K(E^g)$ such that $(u)\cap U=D\cap U$ for some open neighborhood $U$ of $0$. We have the following lemma:
\begin{lem}
$u$ induces a trivialization $\mathcal{O}(D)_{|0}\cong k$. The corresponding evaluation map on multiples of $D$ is:
$$\ev_0:\HH^0(E^g, \mathcal{O}(ND))\rightarrow k, \,, \rho\mapsto (\rho u^{N})(0)$$
\end{lem}
\begin{proof}
Since a section $\rho\in \HH^0(E^g, \mathcal{O}(ND))$ satisfies $(\rho)+ND\geqslant 0$, the rational function $(\rho u^{N})$ is regular in a neighborhood of $0$. Therefore the expression $(\rho u^{N})(0)$ is well-defined. The proof of the lemma is elementary algebraic geometry.
\end{proof}
We are now ready to describe our method for calculating $\vartheta_{a,b}$. To setup the notation recall that $\rho_{\tilde{\mathcal{H}}}\in K(E^g)$ is the rational function from Theorem \ref{ThmInvConn}. Furthermore, recall that we have constructed level subgroups $\tilde{V_1}, \tilde{V_2}\subset \mathcal{G}(\mathcal{O}(4D)$. Notice that $\tilde{V_1}, \tilde{V_2}\cong Z_4$ canonically.
\begin{thm}\label{ThmThetaNull}
For any $a,b\in \mathbb{Z}^g$ denote by $v_a\in \tilde{V_1}$ (resp. $v_b\in \tilde{V_2}$) the element with coordinate vector $a$ (resp. $b$). Then
$$\vartheta_{a,b}=\ev_0\left( U_{v_b} \left( U_{v_a} \left(\rho_{\mathfrak{2}} \cdot [2]^*\left(\rho_{\tilde{\mathcal{H}}}\right)\right)\right)\right)$$
\end{thm}
\begin{proof}
Consider the following diagram:
$$\begin{tikzcd}
\HH^0(A, \mathcal{L}) \arrow{r}{\left[2\right]^*}\arrow{d}{\pi^*} & \HH^0(A, \mathcal{L}^4)\arrow{d}{\pi^*}\\
\HH^0(E^g, \mathcal{O}(D)) \arrow{r}{\rho_{\mathfrak{2}} \cdot \left[2\right]^*(\cdot)} & \HH^0(E^g, \mathcal{O}(4D))
\end{tikzcd}\,.$$
The diagram can be made commutative by choosing the isomorphism $[2]^*\mathcal{L}\cong \mathcal{L}^4$ properly. Now the element $s\in \HH^0(A, \mathcal{L})$ maps to $\rho_{\tilde{\mathcal{H}}}\in \HH^0(E^g, \mathcal{O}(D))$ under $\pi^*$ up to a scalar. 
\par On the other hand by the definition of $\mathcal{G}(\delta)$, the element $v_b\cdot v_a\in \tG{4}(k)$ maps to $(1,a,b)\in \mathcal{G}(\delta)$ under the theta structure $f_2: \tG{4}\rightarrow \mathcal{G}(\delta)$. Here $\delta=(4,\ldots, 4)$.
\par Therefore the element $U_{f_{2,k}^{-1}(1,a,b)}  ([2]^* s)$ maps to $U_{v_b} \left( U_{v_a} \left(\rho_{\mathfrak{2}} \cdot [2]^*\left(\rho_{\tilde{\mathcal{H}}}\right)\right)\right)$ under $\pi^*$.
\par The theorem follows from Theorem \ref{ThetaNullDescr} because the diagram
$$
\begin{tikzcd}
\HH^0(A, \mathcal{L}^4)\arrow{r}{\ev_0} \arrow{d}{\pi^*} & k\\
\HH^0(E^g, \mathcal{O}(4D)) \arrow{ur}{\ev_0}
\end{tikzcd}$$
commutes up to a scalar.
\end{proof}
\begin{rem} In fact the implemented method is slightly different: Instead of using the local equation for $D$ at $0$ we evaluate at another point $Q\in E^g[2]$ such that $Q\notin D$. Furthermore, recall that we defined $D=t_P^{-1}(D_r)$. However, for computational efficiency it is advantageous to omit this translation by $P$. These two modifications will lead to a linear transformation of the theta nullvector that needs to be inverted in the end.
\end{rem}
\begin{rem}\label{RemLev2}
For the computation of hyperelliptic curves it is enough to calculate the theta nullvalues of $(A, \mathcal{L}^2)$ instead of the $\vartheta_{a,b}$. There is a faster method for calculating those instead of computing $\vartheta_{a,b}$. Indeed, recall that $\vartheta_{a,b}$ was related to the theta nullvalues of $(A, \mathcal{L}^4)$. Therefore the computation of the theta nullvalues of $(A, \mathcal{L}^2)$ will involve rational functions of smaller degree (essentially a factor of $2$).
\par In that method one considers the section $(\rho_{\tilde{\mathcal{H}}})^2$ in $\HH^0(E^g, \mathcal{O}(2 D))$ instead of $\left(\rho_{\mathfrak{2}} \cdot [2]^*\left(\rho_{\tilde{\mathcal{H}}}\right)\right)$. However, the disvantage here is that the coordinates of this section in the canonical basis will not be independent of the abelian variety, as it was the case in Lemma \ref{Mult2}. Therefore one needs more than the hypothetical optimal $2^g$ theta group operations to extract the theta nullvalues from that section. Yet, it is still faster to compute level $2$ theta nulls that way than by passing from level $4$ to level $2$ because of the saving in the degrees.
\end{rem}
\begin{rem}\label{RemGenNsq}
The method we presented here generalizes to the computation of the theta nullvalues of $(A, \mathcal{L}^N)$ where $N=n^2$ is a square. This requires a generalization of Lemma \ref{Mult2} with $2$ replaced by $n$.
\par Similarly the method from Remark \ref{RemLev2} generalizes to arbitrary level $N$. But, if $N$ is a square, the method exploiting that will be faster. Both of these algorithms are included in our implementation for arbitrary $N$.
\par We will not give the details of these more general methods, however, because there is the following more efficient alternative for computing the higher level theta nullvalues: You can get the theta nullvalues of level $4 l$ with $l$ odd by combining the present paper with Lubicz \& Robert's \cite{Lubicz-Robert}. For that purpose we consider the map $\phi:E^g\rightarrow \mathbb{P}^{4^{g}}$ defined by the rational functions
$$U_{v_b} \left( U_{v_a} \left(\rho_{\mathfrak{2}} \cdot [2]^*\left(\rho_{\tilde{\mathcal{H}}}\right) \right)\right)\,.$$
By construction $\phi$ factors as the composition $E^g \stackrel{\pi}{\rightarrow} A \stackrel{\iota}{\rightarrow} \mathbb{P}^{4^{g}} $, where $\iota$ is (up to a projective transformation) the closed immersion defined by Mumford \cite[p. 298]{EqDefAV}. Then you can compute the theta nullvalues of level $4l$ from the knownledge of $\theta_{a,b}$ and $\phi(E^g[l])$ with the method of (loc. cit.).
\end{rem}
\subsection{Complexity analysis}
We analyze the complexity of the main algorithm in terms of the following variables:
\begin{itemize}
\item $n$, the positive integer such that $\ker(\eta)\subseteq E^g[F^n]$.
\item $r$, the cardinality of the spanning tuple.
\item $a$, the number of generators of $M(\mathcal{H})$ as a Dieudonn\'e module.
\item $\Vert H \Vert_2$, where the $2$-norm on the standard quaternion vector space is defined to be
$$\Vert \cdot \Vert_2:\mathbb{H}^g \rightarrow \mathbb{R}_{\geqslant 0},\, \Vert (a_1, \ldots, a_g)\Vert_2 = \left(\sum_{i=1}^g a_i \overline{a_i} \right)^{\frac{1}{2}}$$
and the matrix norm is defined as usual.
\end{itemize}
We begin with a lemma that bounds the degrees of the rational functions appearing in the course of the algorithm. For that purpose we use the fact that $K(E^g)=\text{Frac}(R)$ where
$$R=\frac{k[x_1, y_1, \ldots, x_g, y_g]}{\left( y_i^2-f(x_i) : i =1, \ldots, g\right)}\,.$$
Here $y^2=f(x)$ is a Weierstrass equation for $E$. Clearly every element in $R$ has a representative such that each $y_i$ appears with degree at most $1$. We will assume that every element is written in such a form.
\begin{lem}
Let $i\in \{1, \ldots, g\}$ and $\rho\in \HH^0(E^g, \mathcal{O}(D)$ be arbitrary.
\par Then the largest power of the variable $x_i$ appearing in the numerator and the denominator of $\rho$ is bounded by
$$O(g\, 2^g \Vert H \Vert_2)$$
where the implied constants are independent of $g$ and $p$.
\end{lem}
\begin{proof}
To get the bound one investigates the method from \cite[Section 8.5]{APThesis} for the computation of the Riemann-Roch space $\HH^0(E^g, \mathcal{O}(D))$. Indeed, consider the map $\xi_i: E^g\longrightarrow E_i$ as in Definition \ref{DefGrschCompDec}. It follows from \cite[2nd algorithm in Section 8.4]{APThesis} that the degrees of the polynomials defining $\xi_i$ are bounded by $O(\Vert H \Vert_2)$. The bound
$$O(g\, 2^g \Vert H \Vert_2)$$
follows from the construction of the polynomial $r_\text{den}$ in \cite[Section 8.5]{APThesis} which is a common denominator of all the rational functions in $\HH^0(E^g, \mathcal{O}(D))$ and the bound on the degrees of the numerators from \cite[Lemma 8.15]{APThesis}.
\end{proof}
Let $\delta$ denote the exponent of the complexity of solving a linear system of equations. By \cite{Coppersmith-Winograd} one knows that  $2\leqslant \delta\leqslant 2.376$.
\begin{prop}\label{PropCompl}
For fixed $n, g, a$ the \hyperref[AlgMain]{main algorithm} computes the algebraic theta nullvalues with 
$$O \left(r \Vert H \Vert_2 ^{\delta\, g} \right)$$
operations in $\mathbb{F}_{p^2}$ and
$$\tilde{O}\left(r \Vert H \Vert_2^{\log_2(3) g+\frac{1}{2}} p^{\log_2(3) n^2 a} \right)  $$
operations in the ground field $k$ with Karatsuba multiplication or
$$\tilde{O}\left(r \Vert H \Vert_2^{g+\frac{1}{2}} p^{n^2 a}\right) $$
operations with FFT multiplication. Here $\tilde{O}$ means that we neglect powers of $\log(p), \log(\Vert H \Vert_2)$
\end{prop}
\begin{proof}
We analyze the individual steps of the \hyperref[AlgMain]{main algorithm}. The steps (1)-(4) are clearly cheap and can be neglected. In step (5) the bottleneck is solving a linear system of equations to find the numerator of $\rho_i$ which costs
$$O( \Vert H\Vert_2^{\delta\, g} )$$
operations, where $2\leqslant \delta\leqslant 2.376$ is the exponent of the complexity of solving a linear system of equations \cite{Coppersmith-Winograd}. 
The rational functions $\rho_i$ are defined over $\mathbb{F}_{p^2}$ because the curve $E$, the endomorphisms of $E$, and the points in $E[2]$ are defined over $\mathbb{F}_{p^2}$ (for the last assertion notice that $F^2=-p$ acts trivially on $E[2]$). Thus these calculations take place over $\mathbb{F}_{p^2}$ as claimed.
\par Step (6) is usually the most expensive step. It requires calculations in the ring
$$K(E^g) \otimes_k A_{W_{n,n}^a}\,.$$
The trick for perfomance optimization is to notice that after each theta group operation we will land back in
$$\HH^0(E^g, \mathcal{O}(D)) \otimes_k A_{W_{n,n}^a}\,.$$
Therefore we can use polynomial truncation when we perform the compositions of polynomial expressions in this step (see also \cite[Remark 9.18]{APThesis}). If one uses the Brent-Kung algorithm \cite[Theorem 2.2]{Brent-Kung}, this compostion can be computed with
$$\tilde{O}(\Vert H \Vert_2^{\log_2(3) g+\frac{1}{2}} p^{\log_2(3) n^2 a} )  $$
operations in the ground field $k$ with Karatsuba multiplication (resp.\\ $\tilde{O}(\Vert H \Vert_2^{g+\frac{1}{2}} p^{n^2 a}) $ with FFT).
\par After and before the composition we perform a multiplication with a rational function in $K(E^g)$. That costs
$\tilde{O}(\Vert H \Vert_2^{\log_2(3) g} p^{\log_2(3) n^2 a} ) $ (resp.\\ $\tilde{O}(\Vert H \Vert_2^{g+\frac{1}{2}} p^{n^2 a}) $ with FFT).
Now in step (6) we buy $2r$ of these compositions and $4r$ multiplications. We conclude that step (6) takes as many operations in $k$ as in the statement of the proposition.
\par In step (7) we only use operations in $K(E^g)$. To analyze the degrees on hand we notice that translating by a $2$-torsion point can scale up the degree at worst by a constant factor. This gives a cost of $ \tilde{O} \left( \Vert H \Vert_2 ^ {\log_2(3) g} \right) $ (resp. $ \tilde{O} \left( \Vert H \Vert_2^g \right) $ with FFT) operations in $k$ for step (7). This is clearly dominated by step (6). Therefore the proposition follows.
\end{proof}
\section{Examples}
\subsection{Genus $2$}
\begin{exmpl}
We used our method to implement an algorithm that computes Moret-Bailly families. Instead of using the method of Section \ref{SecInvCon} in [ \hyperref[AlgMain]{Main algorithm} step (6)], we used a formula of Moret-Bailly \cite[Theorem 4.18]{AP}. This means that then steps (1)-(6) from the \hyperref[AlgMain]{main algorithm} are then exactly the same as (loc. cit., Algorithm 2 steps (1)-(7)).
\par Our algorithm is able to compute $\vartheta_{a,b}^2$ for the generic fiber of these families. Using a formula of Rosenhain one can compute a hyperelliptic curve from the $\vartheta_{a,b}^2$. Rosenhain's formula a priori is an analytic formula over $\mathbb{C}$, nevertheless these formulas are valid over any algebraically closed field of characteristic not $2$ by \cite[p. 31]{KRR}.
\par But, since Rosenhain's formula determines only the ramification points, the resulting hyperelliptic model is of the form $\mathcal{C}_{\text{gen}}:y^2=c f(x)$ with an undetermined constant $c\in k(t)$. Therefore the curve is only determined up to a quadratic twist. The constant $c$ can be determined up to $c_0\in k^\times$ by noting that the Jacobian of $\mathcal{C}_{\text{gen}}$ has everywhere good reduction. We did not implement that, however.
\par We used our algorithm to compute the (quadratic twist of the) generic fiber of some Moret-Bailly families. The following table indicates the elliptic curve $E$ and the matrix $H$ describing the polarization.
\par To explain the notation: We use that $E$ is the reduction of a CM elliptic curve and $i, \zeta_3$, resp.  $\sqrt{-2}$ denotes the endomorphism obtained by reduction mod $p$. See \cite[Proposition 2.3.1(ii)]{Silverman} for the elliptic curve with CM by $\mathbb{Z}[\sqrt{-2}]$ we used.
\par
\begin{tabular}{c|c|c|c}
$p$ & $E$ & $H$  & Run time\\
\hline
\hline
$3$ & $y^2=x^3-x$ & $\begin{pmatrix}
3 & (1+i) F\\
-(1+i) F & 3
\end{pmatrix}$ & 0.2s\\
\hline
$5$ & $y^2=x^3-1$ & $\begin{pmatrix}
5 & 2F\\
-2F & 5
\end{pmatrix}$  & 0.9s \\
\hline
$7$ & $y^2=x^3-x$ & $\begin{pmatrix}
7 & -\frac{7}{2} i + 2F+\frac{1}{2} iF\\
\frac{7}{2} i - 2F-\frac{1}{2} iF & 7
\end{pmatrix}$& 4.3s\\
\hline
$11$ & $y^2=x^3-x$ & $\begin{pmatrix}
11 & (3+i)F\\
-(3+i)F & 11
\end{pmatrix}$ & 12.9s\\
\hline
$13$ & $y^2=x^3+4x^2+2x$ & $\begin{pmatrix}
13 & 2(1+\sqrt{-2})F\\
-2(1+\sqrt{-2})F & 13
\end{pmatrix}$& 47s\\
\hline
$17$ & $y^2=x^3-1$& $\begin{pmatrix}
17 & 4F\\
-4F & 17
\end{pmatrix}$ & 4min 27s\\
\hline 
$19$ & $y^2=x^3-x$ & $\begin{pmatrix}
19 & (3+3i)F\\
-(3+3i)F & 19
\end{pmatrix}$ & 5min 21s\\
\hline
$23$ & $y^2=x^3-x$ & $\begin{pmatrix}
23 & -\frac{23}{2} i + 4 F + \frac{1}{2} i F\\
\frac{23}{2} i - 4 F - \frac{1}{2} i F & 23
\end{pmatrix}$ & 43min 13s
\end{tabular}
\end{exmpl}
\subsection{Genus $3$}
\begin{exmpl}\label{ExLi-Oort}
We implemented our method in the case $g=3$ and $\ker(\eta)=E^3[p]$.\\
This leads to the $2$-dimensional families considered by Li and Oort \cite{Li-Oort}.
\par We computed some supersingular genus $3$ curves over finite fields. Our algorithm returned the following hyperelliptic curve over $\mathbb{F}_{3^6}$ with supersingular Jacobian after 4s:
$$C: y^2=x(x-1)(x+1)(x-\beta^{208})(x-\beta^{308})(x-\beta^{420})(x-\beta^{520})\,.$$
\par The following plane quartic with supersingular Jacobian was computed in 4s:
$$C=\mathcal{V}( x_1^4+\beta^{270} x_1^3 x_2+ \beta^{337} x_1^2 x_2^2+\beta^{480} x_1 x_2^3+\beta^{420} x_2^4$$
$$+\beta^{120} x_1^3 x_3    +\beta^{118} x_1^2 x_2 x_3
+\beta^{707} x_1 x_2^2 x_3+
\beta^{169 }   x_2^3 x_3$$
$$+\beta^{584 }   x_1^2 x_3^2+
\beta^{285 }   x_1 x_2 x_3^2+
\beta^{652 }   x_2^2 x_3^2+
\beta^{507 }   x_1 x_3^3+
\beta^{710 }   x_2 x_3^3+
\beta^{46  }  x_3^4)\subset \mathbb{P}^2$$
Here $\beta\in \mathbb{F}_{3^6}$ denotes a root of $x^6 + 2x^4 + x^2 + 2x + 2\in \mathbb{F}_3[x]$.
\par We were also able to calculate a $1$-dimensional family of supersingular quartics in characteristic $3$. The computation took 9s. The equation is too complicated to print it here.
\par The computation of the generic fiber of the $2$-dimensional family was aborted after 10 days. In \cite[Example 11.16]{APThesis} it is  described how to obtain the total family anyway.
\end{exmpl}
\begin{rem}
The models of the curves were computed as follows: By \cite{Oort-Ueno} and Torelli's theorem any indecomposable principally polarized abelian threefold is the Jacobian of a non-singular curve $C$, at least after a quadratic extension of the ground field. Corollary \ref{ThetaVan} and the Riemann-Kempf singularity theorem imply that the hyperellipic and non-hyperelliptic case can be distinguished in the same vein as over $\mathbb{C}$: The curve $C$ is hyperelliptic if and only if $\vartheta_{a,b}=0$ for some $a,b\in \mathbb{Z}^g$ with $a^t b=0$. In this case a model for $C$ can be obtained with Takase's formula (this formula generalizes Rosenhain's formula to arbitrary $g$). If $C$ is non-hyperelliptic, then $C$ is a plane quartic. A model for $C$ can be computed with Weber's formula. These formulas are valid over any algebraically closed field of characteristic not $2$, again by \cite[p. 31]{KRR}.
\par In our code we used the implementation of Takase's and Weber's formula of R. Lercier and J. Sijsling \cite{PlaneCM}.
\end{rem}
\begin{rem}
For the computation of families we plugged in an input over a rational function field $k(t)$ into the algorithm. The method is powerful enough to compute the result quickly. This improves on the implementation from \cite{AP} where it was computationally infeasible to compute generic fibers. The culprit was that a generic algorithm to compute the normalization in (loc. cit, Algorithm 2, step 8) was used. However, there are ways to compute the normalization that exploit the special geometric situation at hand (unpublished, nowhere implemented). That leads to a method to compute a model for the generic fiber of a Moret-Bailly family which is slightly faster than computing theta nullvalues. Nevertheless, the asymptotic running time of the last step would only differ by constant factor and thus the overall complexity is the same  by the analysis in Proposition \ref{PropCompl}.
\par But, for the computation of the generic fiber of a Moret-Bailly family none of these two methods is optimal. In fact, the fastest known way to compute Moret-Bailly families is by interpolating the theta nullvalues based at the completely decomposed points, i.e. the points whose fiber is isomorphic to a product of elliptic curves as a principally polarized abelian variety (see \cite[Section 10.6]{APThesis}). That algorithm has complexity $O(p^{3+\varepsilon})$. Indeed, the bottleneck is the enumeration of all the isomorphism classes of supersingular elliptic curves which has the given complexity as a consequence of \cite[Proposition 7.7]{Kirschmer-Voight}.
\par This alternative approach does not generalize to Li-Oort families with $g\geqslant 3$. The main difference between Moret-Bailly and Li-Oort families is that a Moret-Bailly family is parametrized by \emph{all} the maximal isotropic subgroups of $\ker(\eta)$ whereas Li-Oort put some restrictions on them. Thus interpolating Li-Oort families based at the completely decomposed points fails. Indeed, a maximal isotropic subgroup $\mathcal{H} \subseteq \ker(\eta)$ with $E^g/\mathcal{H}$ completely decomposed does not necessarily satisfy the conditions of Li and Oort \cite[Section 3.6]{Li-Oort}. In fact only a small proportion of them will do so because the condition is a closed one.
\par However, notice that for the method of the present paper it was inessential whether the isogenies to a completely decomposed abelian variety satisfied the condition of Li and Oort.
\end{rem}
\pagebreak
\bibliographystyle{plain}


\begin{thebibliography}{10}

\bibitem{Brent-Kung}
Richard Brent and Hsiang-Tsung Kung.
\newblock {Fast Algorithms for Manipulating Formal Power Series}.
\newblock {\em J. ACM 25}, pages 581--595, 1978.

\bibitem{Coppersmith-Winograd}
Don Coppersmith and Shmuel Winograd.
\newblock Matrix multiplication via arithmetic progressions.
\newblock In {\em {Proceedings of the Nineteenth Annual ACM Symposium on Theory
  of Computing}}, page 1–6, 1987.

\bibitem{Demazure}
Michel Demazure.
\newblock {\em Lectures on {$p$}-divisible {G}roups}.
\newblock Springer Lecture Notes in Mathematics, 1972.

\bibitem{Kempf}
George~R. Kempf.
\newblock {Linear Systems on Abelian Varieties}.
\newblock {\em American Journal of Mathematics}, 111(1):65--94, 1989.

\bibitem{PlaneCM}
Pinar K{\i}l{\i}{\c c}er, Hugo Labrande, Reynald Lercier, Christophe
  Ritzenthaler, Jeroen Sijsling, and Marco Streng.
\newblock Plane quartics over {$\mathbb{Q}$} with complex multiplication.
\newblock {\em Acta Arithmetica}, 185:127--156, 2017.

\bibitem{KRR}
Markus Kirschmer, Fabien Narbonne, Christophe Ritzenthaler, and Damien Robert.
\newblock Spanning the isogeny class of a power of an elliptic curve.
\newblock {\em Mathematics of Computation}, 91:401--449, 2022.

\bibitem{Kirschmer-Voight}
Markus Kirschmer and John Voight.
\newblock Algorithmic enumeration of ideal classes for quaternion orders.
\newblock {\em SIAM J. Comput.}, 39(5):1714–1747, 2010.

\bibitem{Li-Oort}
Ke-Zheng Li and Frans Oort.
\newblock {\em Moduli of {S}upersingular {A}belian {V}arieties}.
\newblock Springer Lecture Notes in Mathematics, 1998.

\bibitem{Lubicz-Robert}
David Lubicz and Damien Robert.
\newblock Fast change of level and applications to isogenies.
\newblock {\em Research in Number Theory}, 9, 2022.

\bibitem{Moret-Bailly}
Laurent Moret-Bailly.
\newblock Familles de courbes et de vari\'et\'es ab\'eliennes sur
  $\mathbb{P}^1$, {II} {E}xemples.
\newblock {\em Ast\'erisque}, 86, 1981.

\bibitem{EqDefAV}
David Mumford.
\newblock On the equations defining abelian varieties {I}.
\newblock {\em Inventiones mathematicae}, (1):287--354, 1966.

\bibitem{MumAV}
David Mumford.
\newblock {\em Abelian {V}arieties}.
\newblock Hindustan book agency, 2nd edition, 1974.

\bibitem{TataIII}
David Mumford.
\newblock {\em Tata {L}ectures on {T}heta {III}}.
\newblock Birkhäuser, 1991.

\bibitem{Oort-Ueno}
Frans Oort and Kenji Ueno.
\newblock Principally polarized abelian varieties of dimension two or three are
  {J}acobian varieties.
\newblock {\em Journ. Fac. Sc. Univ. Tokyo}, 20:377–381, 1973.

\bibitem{APThesis}
Andreas Pieper.
\newblock {\em Applications of {M}umford's {T}heory to {S}upersingular
  {A}belian {V}arieties}.
\newblock PhD thesis, Ulm University, available at
  \url{https://oparu.uni-ulm.de/xmlui/handle/123456789/48334}, 2022.

\bibitem{AP}
Andreas Pieper.
\newblock Constructing all genus 2 curves with supersingular jacobian.
\newblock {\em Research in Number Theory}, 8, 2022.

\bibitem{Sekiguchi}
Tsutomu Sekiguchi.
\newblock {On projective normality of abelian varieties II}.
\newblock {\em Journal of the Mathematical Society of Japan}, 29(4):709 -- 727,
  1977.

\bibitem{Silverman}
Joseph Silverman.
\newblock {\em Advanced topics in the arithmetic of elliptic curves}.
\newblock Springer, 1994.

\bibitem{vdGMooAV}
Gerard van~der Geer and Ben Moonen.
\newblock Abelian {V}arieties (draft book).
\newblock \url{https://www.math.ru.nl/~bmoonen/research.html#bookabvar}.

\bibitem{Water}
William Waterhouse.
\newblock {\em Introduction to {A}ffine {G}roup {S}chemes}.
\newblock Springer, 1979.

\end{thebibliography}
\end{document}